\documentclass[reqno,11pt]{article}
\usepackage{mathrsfs}
\usepackage{amssymb}
\usepackage{amsthm}
\usepackage{amsmath}
\usepackage{amsfonts}
\usepackage{enumerate}
\usepackage{bm}

\usepackage{multicol}

\usepackage{subfigure}
\usepackage{graphicx}
\usepackage{setspace,color,wrapfig}
\usepackage[colorlinks,linkcolor=blue,citecolor=cyan]{hyperref}

\numberwithin{equation}{section}
\newtheorem{theorem}{Theorem}[section]
\newtheorem{lemma}[theorem]{Lemma}

\newtheorem{proposition}[theorem]{Proposition}
\newtheorem{problem}[theorem]{Problem}
\theoremstyle{definition}
\newtheorem{definition}[theorem]{Definition}

\theoremstyle{remark}
\newtheorem{remark}[theorem]{Remark}
\begin{document}
\title{Supersonic  flows with a contact discontinuity  to the two-dimensional steady rotating Euler system }
\author{Shangkun Weng\thanks{School of mathematics and statistics, Wuhan University, Wuhan, Hubei Province, 430072, People's Republic of China. Email: skweng@whu.edu.cn}\and Zihao Zhang\thanks{School of mathematics and statistics, Wuhan University, Wuhan, Hubei Province, 430072, People's Republic of China. Email: zhangzihao@whu.edu.cn}}
\date{}
\maketitle
\newcommand{\de}{{\mathrm{d}}}
\def\div{{\rm div\,}}
\def\curl{{\rm curl\,}}
\newcommand{\ro}{{\rm rot}}
\newcommand{\sr}{{\rm supp}}
\newcommand{\sa}{{\rm sup}}
\newcommand{\va}{{\varphi}}
\newcommand{\me}{\mathcal{M}}
\newcommand{\ml}{\mathcal{V}}
\newcommand{\mi}{\mathcal{N}}
\newcommand{\md}{\mathcal{D}}
\newcommand{\mg}{\mathcal{G}}
\newcommand{\mh}{\mathcal{H}}
\newcommand{\mf}{\mathcal{F}}
\newcommand{\ms}{\mathcal{S}}
\newcommand{\mt}{\mathcal{T}}
\newcommand{\mn}{\mathcal{N}}
\newcommand{\mb}{\mathcal{P}}
\newcommand{\mm}{\mathcal{B}}
\newcommand{\mj}{\mathcal{J}}
\newcommand{\mk}{\mathcal{K}}
\newcommand{\my}{\mathcal{U}}
\newcommand{\mw}{\mathcal{W}}
\newcommand{\mq}{\mathcal{Q}}
\newcommand{\ma}{\mathcal{L}}
\newcommand{\mc}{\mathcal{C}}
\newcommand{\mz}{\bm{ \mathscr{E}}}
\newcommand{\n}{\nabla}
\newcommand{\e}{\bar}
\newcommand{\w}{\tilde}
\newcommand{\m}{\omega}
\newcommand{\Z}{\bar{\bm  Z}}
\newcommand{\h}{\bar{\bm \mathcal{H}}}
 \newcommand{\q}{{\rm R}}
 \newcommand{\ov}{{\overline}}
\newcommand{\p}{{\partial}}
\newcommand{\x}{{\varepsilon}}
\renewcommand\figurename{\scriptsize Fig}
\begin{abstract}
This paper concerns the structural stability of supersonic flows with a  contact discontinuity in a finitely long curved   nozzle for the two-dimensional steady compressible rotating  Euler system. Concerning  the effect of Coriolis force, we first establish the existence of supersonic shear flows with a contact discontinuity in the flat nozzle. Then we consider the stability   of these background supersonic shear flows  with a contact discontinuity when  the incoming supersonic  flow  and the upper and lower nozzle walls are  suitably perturbed.  The problem can be formulated as an initial boundary value problem with a contact discontinuity as a  free boundary.  To deal with the free
boundary value problem,  the Lagrangian transformation is introduced to straighten and fix the contact discontinuity. The rotating Euler system is reduced  to a  first order hyperbolic  system for the Riemann invariants.  We design an iteration  scheme and derive  some  estimates for the solution to the hyperbolic  system. Finally, by using the inverse  Lagrangian transformation, we prove the original free boundary problem admits two layers of smooth supersonic flows separated by a smooth contact discontinuity.
\end{abstract}
\begin{center}
\begin{minipage}{5.5in}
\end{minipage}
\end{center}
\section{Introduction }\noindent
\par  In this paper, we are concerned with the structural stability of  supersonic flows with a contact discontinuity governed by the two-dimensional steady  rotating Euler system in a finitely long curved  nozzle. The  two-dimensional steady  rotating Euler system for compressible  inviscid gas  is of the  following form:
\begin{align}\label{1-1}
\begin{cases}
\p_{x_1}(\rho u_1)+\p_{x_2}(\rho u_2)=0,\\
\p_{x_1}(\rho u_1^2)+\p_{x_2}(\rho u_1 u_2)+\p_{x_1} P-\rho u_2=0,\\
\p_{x_1}(\rho u_1u_2)+\p_{x_2}(\rho u_2^2)+\p_{x_2} P+\rho u_1=0,\\
\p_{x_1}(\rho u_1 E+u_1 P)+\p_{x_2}(\rho u_2 E+u_2 P)=0,\\
\end{cases}
\end{align}
where $\bm u=(u_1,u_2)$ is the velocity, $\rho$ is the density, $P$ is the pressure, $E$ is the energy, respectively.  The term $ \rho \bm u^\bot=(-\rho u_2,\rho u_1)^T $ on the second and third equations in \eqref{1-1}
corresponds to the Coriolis force arising from the Earth's rotation. For polytropic gas, the equation of state and the energy are of the form
\begin{equation}\label{1-2}
P= A(S)\rho^{\gamma}, \quad{\rm {and}}\quad E=\frac12|{\bm{u}}|^2+\frac{ P}{(\gamma-1)\rho},
\end{equation}
where  $A(S)= R e^{S}$ and $\gamma\in (1,+\infty)$, $R$ are positive constants.  Denote the Bernoulli's function and the local sonic speed by $B=\frac12|{\bm{u}}|^2+\frac{\gamma P}{(\gamma-1)\rho}$  and $ c(\rho,A)=\sqrt{A\gamma} \rho^{\frac{\gamma-1}{2}}$, respectively. Then  the system \eqref{1-1} is hyperbolic for supersonic flows ($ |\textbf{u}|>c(\rho,A) $), and hyperbolic-elliptic coupled for subsonic flows ($ |\textbf{u}|<c(\rho,A) $).
\par It is well-known that the study of gas flows in nozzles plays a fundamental role in
the operation of turbines, wind tunnels and rockets. In  the book \cite{CF48},  Courant and Friedrichs
first gave a systematic analysis  via nonlinear partial differential
equations for various types of steady inviscid flows in nozzles. The steady inviscid
flow involving  a contact discontinuity is  one of the typical flow patterns.
  Up to now, there have been many
interesting results on steady flows involving  a contact discontinuity in  nozzles. For the subsonic flow in infinite  nozzles,  the stability  of subsonic  contact discontinuity  was established in  \cite{BM09} and \cite{ BP19,PB19} with a Helmholtz decomposition.    The global existence and uniqueness  of the subsonic contact discontinuity with large vorticity  were obtained  in \cite{CHWX19} by the theory of compensated compactness, which is not a perturbation around piecewise constant solutions. For the flow in finitely nozzles, the stability  of subsonic  contact discontinuity has been analyzed in \cite{WZ23} by using the implicit function
theorem.  This idea is motivated by the discussion of the airfoil problem in \cite{CXZ22}. In \cite{CXZ22}, the vortex line attached to the trailing edge  is  a free boundary,   the authors employed the implicit function theorem to
solve this problem.   The stability of supersonic  and transonic flat contact discontinuity for 2-D steady Euler flows  was established in \cite{HFWX19,HFWX21}. The authors in \cite{HFWX19,HFWX21}   employed the Lagrangian transformation and   the  Riemann invariants to reduce the Euler system in the Lagrangian
coordinates to a diagonal form and then design the iteration scheme to derive the estimates for the approximate solutions.    For other related problems which involve the  stability of contact discontinuity, one can see  \cite{ CYK13,CKY13,WY15,WY13,WF15} and the references therein.
\par In this paper, we consider the structural stability of supersonic rotating Euler flows with a  contact discontinuity  in a two-dimensional  finitely long  nozzle under  small perturbations on  the incoming supersonic flow and the upper and lower walls of the nozzle. The rotating Euler system can be regarded as the classical compressible Euler system with the Coriolis force. The influence and importance of rotational effects in fluids has been documented in various contexts, in particular in the geophysical fluids literature
 \cite{GS07,PJ87} or  for the $ \beta $-plane model \cite{EW17,PW18}. The existence of subsonic  flows for the steady rotating Euler system in an infinitely long axisymmetric nozzle was obtained in \cite{ZZH22}. However, there are no works on steady flows with a contact discontinuity for the rotating Euler system. Thus we first investigate the structural stability of supersonic rotating Euler flows with a contact discontinuity.  The idea used here is the characteristic method, which is a powerful tool  in dealing the boundary value problem for the first order quasilinear hyperbolic system.   Different from the previous work  in \cite{HFWX19,HFWX21},
   the background solution is not the constant solution but depends on the vertical variable.  Due to the effect of the Coriolis force, the Riemann invariants are not conserved along the characteristic any more. However, one still can derive the integral representation formula  for the   Riemann invariants. Then further careful analysis can be carried out to  obtain the existence of the solution for  the linearized  hyperbolic  system and   establish the $ C^{1,\alpha} $ estimates  of the solution    by  the characteristics method for all cases depending on the entrance or on the reflection of the characteristic curves by the upper and lower walls or on the contact discontinuity.
\par This paper will be arranged as follows. In Section 2, we  formulate the problem of supersonic rotating Euler
flows with a contact discontinuity in a two-dimensional finitely long curved nozzle and state
the main result. In Section 3, we first  reduce the free boundary value problem into the fixed boundary value problem  by the Lagrangian transformation, then  the Riemann invariants can be employed to diagonalize the rotating Euler system in the  Lagrangian coordinates. Finally, we design an iteration scheme by linearizing the nonlinear fixed boundary value problem near the background solution and  establish the estimate of the eigenvalues  for  the linearized hyperbolic system. In Section 4, we  establish various $ C^{1,\alpha} $ estimates for the  solution to the linear boundary value problem by using the
characteristics method carefully for all cases depending on the entrance or  on the reflection of the characteristic curves by the upper and lower walls or on the contact discontinuity. In Section 5, we construct a map  based on the solution of the linearized problem and use Schauder fixed point theorem to show the existence of a fixed point to the map. Finally we prove the uniqueness of the solution.
\section{ Formulation of the nonlinear problem}\noindent
 \par In this section, we first  give the definition of steady flows  with  a contact discontinuity. Then  we construct a special class of supersonic shear  flows with a straight contact discontinuity in a finitely long flat nozzle. Finally, we give a detailed formulation  of the stability problem for these supersonic shear  flows with a contact discontinuity under the suitable perturbations of boundary conditions and state the main result.

 \subsection{The definition of steady flows  with  a contact discontinuity}\noindent
 \par    Let $ \Omega $ be an open and connected domain. Suppose that a non-self-intersecting $ C^1$-curve  $ \Gamma$ divides $ \Omega $ into two disjoint open subsets $ \Omega^\pm$  such that $ \Omega= \Omega^- \cup\Gamma \cup \Omega^+ $.
Assume that $ \bm U = (\rho, u_1, u_2,P)$ satisfies the following properties:
\begin{enumerate}[(1)]
\item $\bm U \in[L^\infty(\Omega)\cap C^{1}_{loc}(\Omega^\pm)\cap C^0_{loc}(\Omega^\pm\cup \Gamma)]^4$;
\item For any $\eta\in C_0^{\infty}(\Omega)$,
\begin{equation}\label{1-3}
\begin{cases}
\int_{\Omega}(\rho u_1\p_{x_1}\eta+\rho u_2\p_{x_2}\eta)\de \mathbf{x}=0,\\
\int_{\Omega}((\rho u_1^2+P)\p_{x_1}\eta+\rho u_1u_2\p_{x_2}\eta+\rho u_2\eta)\de \mathbf{x}=0,\\
\int_{\Omega}(\rho u_1u_2\p_{x_1}\eta+(\rho u_2^2+P)\p_{x_2}\eta-\rho u_1\eta)\de \mathbf{x}=0,\\
\int_{\Omega}(\rho u_1(E+\frac{P}{\rho})\p_{x_1}\eta+\rho u_2(E+\frac{P}{\rho})\p_{x_2}\eta)\de \mathbf{x}=0.\\
\end{cases}
\end{equation}
\end{enumerate}
\par By integration by parts, we get the Rankine-Hugoniot conditions:
\begin{equation}\label{1-4}
\begin{cases}
n_1[\rho u_1]+n_2[\rho u_2]=0,\\
n_1[\rho u_1^2]+n_1[P]+n_2[\rho u_1u_2]=0,\\
n_1[\rho u_1u_2]+n_2[\rho u_2^2]+n_2[P]=0,\\
n_1[ \rho u_1(E+\frac{P}{\rho})]+n_2[\rho u_2( E+\frac{P}{\rho})]=0.\\
\end{cases}
\end{equation}
 Here $ \mathbf{n} = (n_1, n_2) $  is the unit normal vector to $\Gamma$, and $[F](\mathbf{x}) = F_+(\mathbf{x})- F_-(\mathbf{x})$
denotes the jump across the  curve $\Gamma$ for a piecewise smooth function $ F $.
 Let $ \bm \tau= (\tau_{1},\tau_{2})$ as the unit tangential vector to $\Gamma$, which means that $ \mathbf{n}\cdot{\bm \tau}= 0$. Taking the dot product of $(\eqref{1-4}_2,\eqref{1-4}_3)$ with $ \mathbf{n}$  and $ \bm \tau $ respectively, one has
\begin{equation}\label{1-5}
\rho(\mathbf{u}\cdot\mathbf{n})[\mathbf{u}\cdot \bm \tau]_{\Gamma}= 0, \quad
[\rho(\mathbf{u}\cdot\mathbf{n})^2+P]_{\Gamma}= 0.
\end{equation}
 Assume that $ \rho> 0 $  in $ \bar \Omega $, \eqref{1-5} is divided into  two subcases:
 \begin{itemize}
 \item  $ \mathbf{u}\cdot\mathbf{n} \neq 0 $ and $[\mathbf{u}\cdot \bm \tau]_{\Gamma}=0 $. In this case,  the curve $ \Gamma $ is a shock;
  \item    $ \mathbf{u}\cdot\mathbf{n}=0 $ and $ [P]=0 $ on $\Gamma$. In this case, the curve $ \Gamma $ is  a contact discontinuity.
 \end{itemize}

 \begin{definition}
We define $ \bm U = (\rho, u_1, u_2,P)$  to be a weak solution of the rotating Euler system \eqref{1-1} in $ \Omega $ with a contact discontinuity $\Gamma$ if the the following properties hold:
\begin{enumerate}[(i)]
\item $\Gamma$ is a non-self-intersecting $ C^1$-curve  dividing $\Omega $ into two disjoint open subsets $ \Omega^\pm$  such that $ \Omega = \Omega^- \cup\Gamma \cup \Omega^+ $;
\item $ \bm U $ satisfies $\rm{(1)}$ and $ \rm{(2)} $;
\item $ \rho>0 $ in $ \bar \Omega$;
 \item $(\mathbf{u}|_{\bar \Omega_-\cap\Gamma}-\mathbf{u}|_{\bar \Omega_+\cap\Gamma })\neq \mathbf{0} $
 holds for all $ \mathbf{x}\in\Gamma$;
 \item $ \mathbf{u}\cdot\mathbf{n}=0 $ and $ [P]=0 $ on $\Gamma$.
 \end{enumerate}
 \end{definition}
 \subsection{The background solutions}\noindent
\par In this subsection,  we construct  a special class of supersonic shear flows with a straight contact discontinuity in a finitely long flat nozzle. The  flat nozzle  of the length $ L $ is given by
\begin{equation*}
\Omega_b=\{(x_1,x_2):0<x_1<L,\ -1<x_2<1\}.
\end{equation*}
Let the line $ x_2=0 $ as the background contact discontinuity curve, which divided the domain into two regions:
\begin{equation*}
\Omega_b^{-}=\Omega_b\cap\{ -1<x_2<0\}, \quad \Omega_b^+=\Omega_b\cap\{ 0<x_2<1\}.
\end{equation*}
 \par Concerning  the effect of the Coriolis force,
  we consider two layers  of steady smooth   shear flows.
   First,
  the velocity  of the  bottom and top layers are given by
  \begin{equation*}
  \bm u_b(x_2)=
  \begin{cases}
    \bm u_b^{-}(x_2):=(u_b^{-}(x_2),0),
     \quad {\rm{in}}\quad  \Omega_b^{-},\\
     \bm u_b^{+}(x_2):=(u_b^{+}(x_2),0),
     \quad {\rm{in}}\quad  \Omega_b^+,\\
      \end{cases}
      \end{equation*}
  where $ u_b^{\pm}(x_2)>0 $.
      Next, for given $A_b^{\pm}>0 $ and  $  A_b^{+}\neq A_b^{-} $ and $\bar P>0 $, the pressure of the bottom and  top layers are defined in the following:
     \begin{equation*}
  P_b(x_2)=
  \begin{cases}
    P_b^-(x_2)=A_b^{-}(\rho_b^{-})^{\gamma} (x_2),
     \quad {\rm{in}}\quad  \Omega_b^{-},\\
     P_b^{+}(x_2)=A_b^{+}(\rho_b^{+})^{\gamma} (x_2),
     \quad {\rm{in}}\quad  \Omega_b^+,\\
      \end{cases}
      \end{equation*}
      with
     \begin{equation*}
        P_b^{-}(0)=P_b^{+}(0)=\bar P>0.
     \end{equation*}
         Then it follows from the third equation in \eqref{1-1} that
        \begin{equation*}
        (P_b^\pm)^{\prime}(x_2)=\gamma A_b^{\pm}
        (\rho_b^{\pm})^{\gamma-1}(\rho_b^{\pm})^\prime(x_2)=
 -(\rho_b^{\pm}u_{b}^{\pm})(x_2).
\end{equation*}
Therefore
\begin{equation}\label{2-1}
    \begin{cases}
    \rho_b^{-}(x_2)
    =\left[\left(\frac{ \bar P}{ A_b^{-}}\right)^{\frac{\gamma-1}{\gamma}}-\frac{\gamma-1}
    {  A_b^{-}\gamma}
    \int_{0}^{x_2}u_{b}^{-}(s)\de s
    \right]^{\frac{1}{\gamma-1}},
     \quad {\rm{in}}\quad  \Omega_b^{-},\\
     \rho_b^{+}(x_2)=\left[\left(\frac{ \bar P}{  A_b^{+}}\right)^{\frac{\gamma-1}{\gamma}}
     -\frac{\gamma-1}
     {  A_b^{+}\gamma}
     \int_{0}^{x_2}u_{b}^{+}(s)\de s\right]^{\frac{1}{\gamma-1}},
     \quad {\rm{in}}\quad  \Omega_b^{+}.\\
      \end{cases}
      \end{equation}
 The functions $ u_{b}^\pm $ and constants $ A_b^\pm $ and $ \bar P $ are chosen such that the shear  flows in the top and bottom layers are supersonic, i.e.,
 \begin{equation*}
 0<\gamma(A_b^{\pm})^{\frac{1}{\gamma}}\bar P^{\frac{\gamma-1}{\gamma}}
 -(\gamma-1)
    \int_{0}^{x_2}u_{b}^{\pm}(s)\de s<(u_{b}^\pm)^2.
\end{equation*}
 Then
\begin{equation}\label{2-2}
 \bm{U}_b(x_2)=
  \begin{cases}
   \bm{U}_b^{-}(x_2):
   =(\rho_b^{-}, u_{b}^{-},0,P_b^{-})(x_2),  \quad {\rm{in}}\quad  \Omega_b^{-},\\
  \bm{U}_b^{+}(x_2):
  =(\rho_b^{+}, u_{b}^{+},0,P_b^{+})(x_2), \quad {\rm{in}}\quad  \Omega_b^{+},\\
  \end{cases}
 \end{equation}
  with a contact discontinuity $ x_2=0 $ satisfy the rotating Euler system \eqref{1-1} in the sense of Definition {\rm{2.1}},
  which will be called the background solutions in this paper. This paper is going to establish the structural
stability of these background  solutions under   the perturbations of suitable boundary conditions on the entrance  and the upper and lower nozzle walls.
 \subsection{The stability problem and the main result}\noindent
\par  The two-dimensional finitely long curved  nozzle $ \Omega $  is  described by
\begin{equation}\label{2-3}
\Omega=\{(x_1,x_2):0<x_1<L, g^-(x_1)<x_2<g^+(x_1)\},
\end{equation}
  with
  \begin{equation*}
  g^\pm(x_1)\in C^{2,\alpha}([0,L]) \quad {\rm{and}}\quad -1\leq g^-(x_1)<g^+(x_1)\leq1.
  \end{equation*}
  The  upper and lower  boundaries of the nozzle are denoted by $ \Gamma_w^+ $ and $ \Gamma_w^- $, i.e;
\begin{equation}\label{2-4}
\Gamma_w^\pm=\{(x_1,x_2):  0<x_1<L,\ x_2=g^\pm(x_1)\}.
\end{equation}
The entrance of the nozzle  is separated into two parts:
\begin{equation}\label{2-5}
\begin{aligned}
\Gamma_{0}^+=\{(x_1,x_2):   x_1=0,\ 0<x_2<g^+(0)\},   \quad \Gamma_{0}^-=\{(x_1,x_2):  x_1=0,\ g^-(0)<x_2<0\}.
\end{aligned}
\end{equation}
 The incoming flow  $ \bm{U}_{0} $ at the entrance is given by
\begin{equation}\label{2-6}
\bm{U}_{0}(x_2)=
\begin{cases}
 \bm {U}_{0}^{+}(x_2):=(\rho_{0}^{+},u_{10}^{+},
u_{20}^{+},P_{0}^{+})(x_2),  \quad {\rm{on}}\quad \Gamma_{0}^{+},\\
\bm{U}_{0}^{-}(x_2):=(\rho_{0}^{-},u_{10}^{-},
u_{20}^{-},P_{0}^{-})(x_2),  \quad {\rm{on}}\quad \Gamma_{0}^{-},\\
  \end{cases}
 \end{equation}
with
 \begin{equation}\label{2-7}
 \frac{u_{20}^{+}}{u_{10}^{+}}(0)=\frac{u_{20}^{-}}{u_{10}^{-}}(0), \quad P_{0}^{+}(0)=P_{0}^{-}(0).
\end{equation}
Here $ \bm{U}_{0} $ has a jump at $ x_2=0 $ and is close to $ \bm{U}_b $ in some sense that will be clarified later.
\par  We expect the flow in the nozzle  will be separated by a  contact discontinuity   $ \Gamma=\{x_2=g_{cd}(x_1),0<x_1<L\} $ with $ g_{cd}(0)=0 $, and we denote
\begin{equation*}
\Omega^+=\Omega\cap\{g_{cd}(x_1)<x_2<g^+(x_1)\}, \ \Omega^-=\Omega\cap\{ g^-(x_1)<x_2<g_{cd}(x_1)\}.
\end{equation*}
 Let
\begin{equation}\label{2-8}
{\bm {U}}(x_1,x_2)=
  \begin{cases}
   {\bm {U}}^+(x_1,x_2):=(\rho^+,u_1^+,u_{2}^+,P^+)(x_1,x_2)\quad {\rm{in}}\quad   \Omega^+,\\
   {\bm {U}}^-(x_1,x_2):=(\rho^-,u_1^-,u_{2}^-,P^-)(x_1,x_2)  \quad {\rm{in}}\quad  \Omega^-.\\
  \end{cases}
 \end{equation}
  Along the contact discontinuity $ x_2=g_{cd}(x_1) $, the following Rankine-Hugoniot conditions hold:
\begin{equation}\label{2-9}
\frac{u_2^+}{u_1^+}=\frac{u_2^-}{u_1^-}=g_{cd}^\prime(x_1), \quad P^+=P^-, \quad {\rm{on}}\quad \Gamma.
\end{equation}
On the nozzle walls $ \Gamma_w^+$  and $ \Gamma_w^- $, the usual slip boundary condition is imposed:
\begin{equation}\label{2-10}
  \frac{u_2^+}{u_1^+}=(g^+)^\prime(x_1), \ {\rm{on}}\ \Gamma_w^+, \quad
  \frac{u_2^-}{u_1^-}=(g^-)^\prime(x_1), \ {\rm{on}}\ \Gamma_w^-.
\end{equation}
\par Under the above setting, we will study the following problem:
\begin{problem}\label{problem}
  Given   supersonic incoming flows  $ \bm{U}_{0}^\pm(x_2) $ at the entrance satisfying \eqref{2-7}, find a unique piecewise smooth supersonic solution $ (\bm{U}^+ ,\bm{U}^-) $  defined on $ \Omega^+ $ and $ \Omega^- $ respectively,  with the contact discontinuity $ \Gamma: x_2=g_{cd}(x_1) $ satisfying the rotating Euler system \eqref{1-1} in the sense of Definition {\rm{2.1}} and  the Rankine-Hugoniot conditions in \eqref{2-9} and the slip boundary conditions in \eqref{2-10}.
 \end{problem}
 \par The main theorem of this paper can be stated as follows.
  \begin{theorem}
   Given functions $\bm U_0^\pm(x_2) $ at the entrance satisfying \eqref{2-7}, we
define
 \begin{equation}\label{2-11}
   \begin{aligned}
   \sigma(\bm U_{0}^+,\bm U_{0}^-,g^+,g^-):&=
   \|\bm U_{0}^+ -\bm U_b^+\|_{1,\alpha;[0,1]}+\|\bm U_{0}^- -\bm  U_b^-\|_{1,\alpha;[-1,0]} \\
   &\quad+\| g^+ - 1\|_{2,\alpha;[0,L]}+
   \| g^- + 1\|_{2,\alpha;[0,L]}.
   \end{aligned}
   \end{equation}
   There exist positive constants $\sigma_1 $ and $ \mc_1 $ depending only on  $ (\bm U_b^+,\bm U_b^-,L,\alpha) $ such that  if
   \begin{equation}\label{2-12}
   \begin{aligned}
   \sigma(\bm U_{0}^+,\bm U_{0}^-,g^+,g^-)\leq \sigma_1,
   \end{aligned}
   \end{equation}
   $\mathbf{Problem \ 2.2}$ has a unique piecewise smooth supersonic flow  $ (\bm{U}^+,\bm{U}^-) $   with the contact discontinuity $ \Gamma: x_2=g_{cd}(x_1) $ satisfying the following properties:
 \begin{enumerate}[\rm(i)]
\item The piecewise smooth supersonic flow  $ (\bm{U}^+,\bm{U}^-)\in ( C^{1,\alpha}(\overline{\Omega^+}))\times (C^{1,\alpha}(\overline{\Omega^-}))$  satisfies the following estimate:
 \begin{equation}\label{2-13}
\|\bm{U}^+ -\bm{U}_b^+\|_{C^{1,\alpha}(\overline{\Omega^+})}+\|\bm{U}^- -\bm{U}_b^-\|_{C^{1,\alpha}(\overline{\Omega^-})}\leq \mc_1\sigma(\bm U_{0}^+,\bm U_{0}^-,g^+,g^-).
\end{equation}
\item The contact discontinuity curve $ g_{cd}(x_1)\in C^{2,\alpha}([0,L])  $ satisfies
    $g_{cd}(0)=0 $. Furthermore, it holds that
\begin{equation}\label{2-14}
 \|g_{cd} \|_{2,\alpha;[0,L]}\leq \mc_1\sigma(\bm U_{0}^+,\bm U_{0}^-,g^+,g^-).
 \end{equation}
 \end{enumerate}
 \end{theorem}
 \section{The reformulation of Problem 2.2}\noindent
\par In this section, we first  reduce the free boundary value problem 2.2 to the fixed boundary value problem 3.1 by the Lagrangian transformation, then the Riemann invariants can be employed to diagonalize the rotating Euler system in the  Lagrangian coordinates. Finally, we design an iteration scheme by linearizing the nonlinear fixed boundary value problem near the background solution and establish  the estimate of the eigenvalues for the linearized  hyperbolic system.
\subsection{Reformulation by the Lagrangian transformation}\noindent
 \par Since the tangent of the contact discontinuity $ \Gamma $ is parallel to the velocity of the flow on its both sides, it is  convenient to apply the Lagrangian transformation to fix the free boundary $ \Gamma $ and  reformulate $\mathbf{ Problem \ 2.2} $ into a fixed boundary value problem in the Lagrangian coordinates.
 \par  Let $ (\bm{U}^{-}, \bm{U}^{+}, g_{cd}(x_1)) $ be a solution to  $\mathbf{ Problem \ 2.2} $.
 Define
 \begin{equation}\label{3-1}
 m^{-}=\int_{g^-(0)}^{0}\rho_{0}^{-}u_{10}^{-}(s)\de s \quad {\rm{and}}\quad
  m^{+}=\int_{0}^{g^+(0)}
  \rho_{0}^{+}u_{10}^{+}(s)\de s.
  \end{equation}
 Then for any $x_1\in(0,L)$, it follows from the conservation of mass equation in \eqref{1-1} that
 \begin{equation}\label{3-2}
 \int_{g^-(x_1)}^{g_{cd}(x_1)}\rho^{-} u_{1}^{-}(x_1,s)\de s=m^{-}, \quad
 \int_{g_{cd}(x_1)}^{g^+(x_1)}\rho^{+} u_{1}^{+}(x_1,s)\de s =m^{+}.
 \end{equation}
 For $ (x_1,x_2)\in \Omega $,  define
\begin{equation}\label{3-3}
y_1=x_1, \quad  y_2(x_1,x_2)= \int_{g^-(x_1)}^{x_2}\rho u_{1}(x_1,s)\de s-m^{-}.
  \end{equation}
    It is easy to verify that
  \begin{equation*}
  \frac{\p y_2}{\p x_1}=-\rho u_2, \quad  \frac{\p y_2}{\p x_2}=\rho u_1.
  \end{equation*}
  Furthermore, a direct computation gives
   \begin{equation*}
   \frac{\p(y_1,y_2)}{\p(x_1.x_2)}=\left|
\begin{matrix} 1& 0 \\ -\rho u_2 & \rho u_1\end{matrix}\right|=\rho u_1.
 \end{equation*}
 So  if $ (\rho^\pm,u_1^\pm, u_2^\pm,P^\pm) $ are close to the background solutions
$ (\rho_b^\pm,u_{b}^\pm,0,P_b^\pm) $, we have $ \rho^\pm u_1^\pm> 0 $. Hence the  Lagrangian transformation is invertible.
\par Under this transformation,  the domain $ \Omega $ becomes
\begin{equation*}
\mn=\{(y_1,y_2):0<y_1<L,\ -m^{-}<y_2<m^{+}\}.
\end{equation*}
The lower wall $ \Gamma_w^-$ and the upper wall $ \Gamma_w^+$  are  straightened into
\begin{equation}\label{3-5}
\begin{aligned}
\Sigma_w^-=\{(y_1,y_2): 0<y_1<L,\ y_2=-m^{-}\}, \quad
\Sigma_w^+=\{(y_1,y_2): 0<y_1<L,\ y_2=m^{+}\}.
\end{aligned}
\end{equation}
Meanwhile, the entrance of the domain is  transformed  into
\begin{equation}\label{3-6}
\begin{aligned}
\Sigma_{0}^{+}=\{(y_1,y_2):  y_1=0,\ 0<y_2<m^{+}\}, \quad
\Sigma_{0}^{-}=\{(y_1,y_2):  y_1=0,\ -m^{-}<y_2<0\}.
\end{aligned}
\end{equation}
Moreover, on $ \Gamma $, one has
\begin{equation*}
y_2(x_1,g_{cd}(x_1))=\int_{g^-(x_1)}^{g_{cd}(x_1)}\rho u_{1}(x_1,s)\de s-m^{-}=0.
\end{equation*}
Hence the free boundary $\Gamma $ becomes the following fixed straight line
\begin{equation}\label{3-7}
\Sigma=\{(y_1,y_2): 0<y_1<L,\ y_2=0\}.
\end{equation}
\par Define
\begin{equation*}
 \mn^{+}=\mn\cap \{0<y_2<m^{+}\},\quad \mn^{-}=\mn\cap \{-m^{-}<y_2<0\}.
\end{equation*}
Let
\begin{equation*}
{\bm{U}}(y_1,y_2)=
  \begin{cases}
   {\bm{U}}^{+}(y_1,y_2):= (\rho^{+},u_1^{+},u_{2}^{+},P^{+})(\mathbf{x}(y_1,y_2)),
     \quad  {\rm{in}}\quad   \mn^{+},\\
     {\bm{U}}^{-}(y_1,y_2):=(\rho^{-},u_1^{-},
   u_{2}^{-},P^{-})(\mathbf{x}(y_1,y_2)),  \quad {\rm{in}}\quad  \mn^{-}.\\
  \end{cases}
 \end{equation*}
Then the rotating Euler system \eqref{1-1} in the new coordinates  can be written as
\begin{equation}\label{3-8}
\begin{cases}
\p_{y_1}\left(\frac{1}{\rho u_1}\right)-\p_{y_2}\left(\frac{ u_2}{u_1}\right)
=0,\\
\p_{y_1} u_2+\p_{y_2}P+1=0,\\
\p_{y_1} u_1+\frac{1}{\rho u_1}(\p_{y_1}-\rho u_2\p_{y_2})P-\frac{u_2}{u_1}=0,\\
\p_{y_1}B=0.\\
\end{cases}
\end{equation}
\par  In  the new coordinates $ (y_1,y_2) $,   one derives
\begin{equation*}
  \frac{\p x_2}{\p y_1}=\frac{u_2}{ u_1},\quad  \frac{\p x_2}{\p y_2}=\frac{1}{\rho u_1}, \quad x_2(y_1,-m^{-})=g^-(y_1).
  \end{equation*}
  Then it holds that
\begin{equation}\label{3-9}
x_2(y_1,y_2)
=\int_{-m^{-}}^{y_2}\frac{1}{\rho u_{1}}(y_1,s)\de s+g^-(y_1).
\end{equation}
In particular, for the background solution $ (\rho_b^\pm,u_{b}^\pm,0,P_b^\pm) $,  one has
\begin{equation}\label{3-10}
x_{2b}(y_2)
=\int_{-m_b^{-}}^{y_2}\frac{1}{\rho_bu_{b}}(s)\de s-1.
\end{equation}
That is
\begin{equation*}
x_{2b}(y_2)=
\begin{cases}
\begin{aligned}
&x_{2b}^{-}(y_2)
=\int_{-m_b^{-}}^{y_2}\frac{1}{\rho_b^{-} u_{b}^{-}}(s)\de s-1, \quad -m_b^{-}<y_2<0,\\
&x_{2b}^{+}(y_2)
=\int_{0}^{y_2}\frac{1}{\rho_b^{+} u_{b}^{+}}(s)\de s, \quad\quad\quad\quad 0<y_2<m_b^{+},\\
\end{aligned}
\end{cases}
\end{equation*}
where $ m_b^{-}=\int_{-1}^{0}\rho_b^{-}u_{b}^{-}(s)\de s $ and $m_b^{+}=\int_{0}^{1}\rho_b^{+} u_{b}^{+}(s)\de s $. For convenience, we assume that
 $  m^{\pm}=m_b^{\pm} $.
Therefore the background solutions in the  Lagrange coordinates are
\begin{equation}\label{3-11}
\w {\bm{U}}_b(y_2)=
\begin{cases}
\begin{aligned}
  &\w {\bm{U}}_b^{-}(y_2)=(\w \rho_b^{-},\w u_{b}^{-},0, \w P_b^{-})(y_2)=(\rho_b^{-},u_{b}^{-},0,P_b^{-})( x_{2b}^{-}(y_2)), \quad {\rm{in}} \quad \mn^{-},\\
  &\w {\bm{U}}_b^{+}(y_2)=(\w \rho_b^{+},\w u_{b}^{+},0, \w P_b^{+})(y_2)=(\rho_b^{+},u_{b}^{+},0,P_b^{+})( x_{2b}^{+}(y_2)), \quad {\rm{in}} \quad \mn^{+}.\\
  \end{aligned}
  \end{cases}
 \end{equation}
 \par At the entrance, the incoming flow is given by
\begin{equation}\label{3-12}
\bm{U}_{0}(y_2)=
\begin{cases}
\bm{U}_{0}^{-}(y_2)=
(\w\rho_{0}^{-},\w U_{10}^{-},\w U_{20}^{-},\w P_{0}^{-})(y_2)  \quad {\rm{on}}\quad \Sigma_{0}^{-},\\
 \bm{U}_{0}^{+}(y_2)=
(\w\rho_{0}^{+},\w U_{10}^{+},\w U_{20}^{+},\w P_{0}^{+})(y_2)  \quad {\rm{on}}\quad \Sigma_{0}^{+},\\
  \end{cases}
 \end{equation}
 where
 \begin{equation*}
 (\w\rho_{0}^{-},\w U_{10}^{-},\w U_{20}^{-},\w P_{0}^{-})(y_2)
 =(\rho_{0}^{-},U_{10}^{-},U_{20}^{-}, P_{0}^{-})\left(\int_{-m^{-}}^{y_2}\frac{1}{\rho_0^- u_{10}^-}(s)\de s+g^-(0)\right),
 \end{equation*}
 and
 \begin{equation*}
 (\w\rho_{0}^{+},\w U_{10}^{+},\w U_{20}^{+},\w P_{0}^{+})(y_2)
 =(\rho_{0}^{+},U_{10}^{+},U_{20}^{+}, P_{0}^{+})\left(\int_{0}^{y_2}\frac{1}{\rho_0^+ u_{10}^+}(s)\de s\right).
 \end{equation*}
 Moreover, it follows from the fourth equation in \eqref{3-8} that
 \begin{equation}\label{3-43}
B(y_1,y_2)=\w B_{0}(y_2)=\begin{cases}
\w B_{0}^{-}(y_2), \quad {\rm{in}} \quad \mn^{-},\\
\w B_{0}^{+}(y_2), \quad {\rm{in}} \quad \mn^{+}.\\
\end{cases}
\end{equation}
Here $ \w B_{0}^\pm=\frac12 (|\w U_{10}^{\pm}|^2+|\w U_{20}^{\pm}|^2)+\frac{\gamma\w P_{0}^\pm}{(\gamma-1)\w \rho_{0}^\pm} $.
\par The Rankine-Hugoniot conditions in \eqref{2-9}  become
\begin{equation}\label{3-13}
\frac{u_2^{+}}{u_1^-}(y_1,0)
=\frac{u_2^-}{u_1^-}(y_1,0),\quad
P^+(y_1,0)=P^-(y_1,0), \ {\rm{on}}\quad \Sigma.
\end{equation}
The slip boundary conditions in \eqref{2-10} can be written as
\begin{equation}\label{3-14}
\frac{u_2^+}{u_1^+}(y_1,m^+)=(g^+)^{\prime}(y_1), \ {\rm{on}}\quad \Sigma_w^+, \quad \frac{u_2^-}{u_1^-}(y_1,-m^-)=(g^-)^{\prime}(y_1), \ {\rm{on}}\quad \Sigma_w^-.
\end{equation}
\par Thus the free boundary value problem, $ \mathbf{Problem \ 2.2}$, in the Euler coordinates can be reformulated as the following problem in the Lagrangian coordinates.
\begin{problem}
     Given   supersonic incoming flows  $ \bm{U}_{0}^\pm(x_2) $ at the entrance satisfying \eqref{2-7}, find a unique piecewise smooth supersonic solution $ (\bm{U}^+ ,\bm{U}^-) $  defined on $ \mn^+ $ and $ \mn^- $ separated by the straight line $ \Sigma $ satisfying the rotating Euler system \eqref{3-8} in the sense of weak solution and  the Rankine-Hugoniot conditions in \eqref{3-13} and the slip boundary conditions in \eqref{3-14}.
 \end{problem}
 \par Before we state  the main result  in the Lagrangian coordinates, let us introduce the following notation: a constant $ C $ is said to depend on the data if  it depends on $ (\w{\bm U}_b^+,\w{\bm U}_b^-,L,\alpha) $. Then Theorem 2.2 will   follows from the following theorem:
  \begin{theorem}
    There exist positive constants $\sigma_2 $ and $ \mc_2 $ depending only on the data such that  if
   \begin{equation}\label{3-15}
   \begin{aligned}
   \sigma(\bm U_{0}^+,\bm U_{0}^-,g^+,g^-)\leq \sigma_2,
   \end{aligned}
   \end{equation}
   $\mathbf{Problem \ 3.1} $ has a unique piecewise smooth supersonic flow  $ (\bm{U}^+,\bm{U}^-) $    satisfying
 \begin{equation}\label{3-16}
 \|\bm{U}^{+} -\w{\bm{U}}_b^+\|_{1,\alpha;\overline{\mn^+}}+ \|\bm{U}^{-} -\w{\bm{U}}_b^-\|_{1,\alpha;\overline{\mn^-}}
\leq \mc_2  \sigma(\bm U_{0}^+,\bm U_{0}^-,g^+,g^-).
\end{equation}
\end{theorem}
\begin{remark}
If Theorem 3.2 holds, then it follows from \eqref{3-9} that
\begin{equation*}
g_{cd}(y_1)=\int_{-m^{-}}^{0}\frac{1}{\rho^- u_{1}^-}(y_1,s)\de s+g^-(y_1),
\quad y_1\in[0,L].
\end{equation*}
Obviously,
$g_{cd}^{\prime}(y_1)=\frac{u_2}{u_1}(y_1,g_{cd}(y_1))\in C^{1,\alpha}([0,L])$. Therefore in the rest of paper, we only consider $\mathbf{ Problem \ 3.1} $ and prove Theorem 3.2.
\end{remark}
  \subsection{ Riemann invariants}\noindent
\par First, we rewrite \eqref{3-8} as the following first-order non-divergence symmetric system:
\begin{equation}\label{3-17}
A(\bm{U})\p_{y_1}\bm{U}+B(\bm{U})\p_{y_2}\bm{U}
=C(\bm{U}),
\end{equation}
where $ \bm{U}=(u_1,u_2,P) $ and
\begin{equation*}
A(\bm{U})=\left(\begin{array}{ccc}
1 & 0 & \frac{1}{\rho u_1}\\
0 & 1 & 0 \\
 \frac{1}{\rho u_1}& 0  & \frac{1}{\rho^2 c^2}
 \end{array}\right)
  \quad B(\bm{U})=\left(\begin{array}{ccc}
0 & 0 &\frac{-u_2}{ u_1}\\
0 & 0 & 1 \\
 \frac{-u_2}{ u_1}& 1  & 0
\end{array}\right)
\quad C(\bm{U})=\left(\begin{array}{ccc}
\frac{u_2}{ u_1} \\
 -1 \\
 0\\
\end{array}\right).
 \end{equation*}
The eigenvalues of \eqref{3-17} are
\begin{equation*}
\lambda_1=\frac{\rho u_1c^2}{u_1^2-c^2}\left(\frac {u_2} {u_1}-
\frac{\sqrt{u_1^2+u_2^2-c^2}}{c}\right),\ \ \lambda_0=0,\ \  \lambda_2=\frac{\rho u_1c^2}{u_1^2-c^2}\left(\frac {u_2} {u_1}+
\frac{\sqrt{u_1^2+u_2^2-c^2}}{c}\right),
\end{equation*}
and the associated right eigenvectors are
\begin{equation*}
r_1=\left(-\frac{\lambda_1}{\rho u_1}-\frac {u_2} {u_1},1,\lambda_1\right)^\top,
\ \ r_0=(u_1,u_2,0)^\top, \ \ r_2=\left(-\frac{\lambda_2}{\rho u_1}-\frac {u_2} {u_1},1,\lambda_2\right)^\top.
\end{equation*}
Let
\begin{equation*}
l_i=(r_i)^\top, \quad i=1,2, \quad l_0=(r_0)^\top.
\end{equation*}
Multiplying system \eqref{3-17} by $ l_i=(r_i)^\top $ and $ l_0 $ to get
\begin{equation}\label{3-18}
\begin{cases}
\begin{aligned}
& {u_2} (\p_{y_1}u_1+\lambda_1\p_{y_2}u_1)-{u_1}(\p_{y_1}u_2+\lambda_1\p_{y_2}u_2)
+\frac{\sqrt{u_1^2+u_2^2-c^2}}{\rho c}(\p_{y_1}P+\lambda_1\p_{y_2}P)\\
&=\left(\frac{\lambda_1}{\rho u_1}+\frac {u_2} {u_1}\right) {u_2} + {u_1},\\
& {u_2} (\p_{y_1}u_1+\lambda_2\p_{y_2}u_1)-{u_1}(\p_{y_1}u_2+\lambda_2\p_{y_2}u_2)
-\frac{\sqrt{u_1^2+u_2^2-c^2}}{\rho c}(\p_{y_1}P+\lambda_2\p_{y_2}P)\\
&=\left(\frac{\lambda_2}{\rho u_1}+\frac {u_2} {u_1}\right) {u_2} + {u_1},
\end{aligned}
\end{cases}
\end{equation}
and
\begin{equation}\label{3-19}
u_1\p_{y_1}u_1+u_2\p_{y_1}u_2+\frac{\p_{y_1}P}{\rho}=0.
\end{equation}
\par  By the  Bernoulli's law, \eqref{3-19} can be rewritten as
\begin{equation*}
\p_{y_1} A=0,
\end{equation*}
which implies
\begin{equation}\label{3-20}
A(y_1,y_2)=\w A_{0}(y_2)=\begin{cases}
\w A_{0}^{-}(y_2), \quad {\rm{in}} \quad \mn^{-},\\
\w A_{0}^{+}(y_2), \quad {\rm{in}} \quad \mn^{+}.\\
\end{cases}
\end{equation}
Here $ \w A_{0}^\pm=\frac{\w P_{0}^\pm}{(\w \rho_{0}^\pm)^\gamma} $.
 Denote
\begin{equation*}
W=\frac{u_2}{u_1},\quad J=\frac{\sqrt{u_1^2+u_2^2-c^2}}{\rho c u_1^2}.
\end{equation*}
 Then it follows from the Bernoulli's law  that $ u_1, u_2 $ can be expressed as
\begin{equation}\label{3-22}
u_1=
\sqrt{\frac{2\left((\gamma-1)\w B_{0}-\gamma(\w A_{0})^{\frac{1}{\gamma}}
P^{\frac{\gamma-1}{\gamma}}\right)}{(\gamma-1)(1+W^2)}},\ \
u_2=
W\sqrt{\frac{2\left((\gamma-1)\w B_{0}-\gamma(\w A_{0})^{\frac{1}{\gamma}}
P^{\frac{\gamma-1}{\gamma}}\right)}{(\gamma-1)(1+W^2)}}.
\end{equation}
Thus \eqref{3-22} implies that we only need to solve $ W $ and $ P $. \par First,
 \eqref{3-18} can be written as
\begin{equation}\label{3-21}
\begin{cases}
(\p_{y_1}+\lambda_1\p_{y_2})W-J(\p_{y_1}+\lambda_1\p_{y_2})P=
-\left(\frac{\lambda_1}{\rho u_1}+\frac {u_2} {u_1}\right)\frac {u_2}{u_1^2} - \frac {1}{u_1}, \\
(\p_{y_1}+\lambda_2\p_{y_2})W+J(\p_{y_1}+\lambda_2\p_{y_2})P=
-\left(\frac{\lambda_2}{\rho u_1}+\frac {u_2} {u_1}\right)\frac {u_2}{u_1^2} - \frac {1}{u_1}.
\end{cases}
\end{equation}
Set $ G=(W,P)^\top $. Then \eqref{3-21} can be rewritten as the  following first-order  symmetric system:
\begin{equation}\label{3-23}
\p_{y_1}G+M\p_{y_2}G=N,
 \end{equation}
 where
 \begin{equation*}
M=\left(\begin{array}{ccc}
\frac{\rho c^2 u_2}{u_1^2-c^2}& \frac{u_1^2+u_2^2-c^2}{u_1(u_1^2-c^2)} \\
\frac{\rho^2 c^2 u_1^3}{u_1^2-c^2} & \frac{\rho c^2 u_2}{u_1^2-c^2} \\
 \end{array}\right),\quad
 N=\left(\begin{array}{ccc}
-\frac{u_1^2+u_2^2-c^2}{u_1(u_1^2-c^2)}   \\
-\frac{\rho c^2 u_2}{u_1^2-c^2}   \\
\end{array}\right).
 \end{equation*}
 A direct computation yields that the eigenvalues of \eqref{3-23} are $ \lambda_i $  $ (i=1,2) $ and the corresponding right eigenvectors are
  \begin{equation*}
  \w r_1=(\sqrt{u_1^2+u_2^2-c^2},-\rho cu_1^2)^\top, \quad  \w r_2=(\sqrt{u_1^2+u_2^2-c^2},\rho cu_1^2)^\top.
   \end{equation*}
   Define the Riemann invariants $ Z_i $ $ (i=1,2) $ for the system \eqref{3-23} as
   \begin{equation}\label{3-24}
   Z_1=\arctan W-\Lambda( P,\w B_{0},\w A_{0}),\quad
   Z_2=\arctan W+\Lambda(P,\w B_{0},\w A_{0}),
   \end{equation}
   where
   \begin{equation}\label{3-25}
  \Lambda(P,\w B_{0},\w A_{0})=\int^{P}\frac{\sqrt{2\w B_{0}-\frac{\gamma(\gamma+1)}
  {\gamma-1}
  \w A_{0}^{\frac{1}{\gamma-1}}s^{\frac{\gamma-1}{\gamma}}}}
  {2\gamma^{\frac12} \w A_{0}^{-\frac{1}{2\gamma}}\left(\w B_{0}-\frac{\gamma}{\gamma-1}
  \w A_{0}^{\frac{1}{\gamma}}s^{1-\frac{1}{\gamma}}\right)s^{\frac{\gamma+1}{2\gamma}}}\de s.
   \end{equation}
  Then it follows from \eqref{3-24} and \eqref{3-25} that
  \begin{equation}\label{3-26}
  W=\tan\left(\frac{Z_1+Z_2}{2}\right), \quad \Lambda(P,\w B_{0},\w A_{0})
  =\frac12(Z_2-Z_1).
   \end{equation}
  \par Set $  \bm Z=(Z_1,Z_2)^\top $.
   A direct  computation yields that
    \begin{equation}\label{3-27}
    \begin{aligned}
    &\frac{\p P}{\p Z_1}=-\frac{1}{2\p_P\Lambda(P,\w B_{0},\w A_{0})}, \quad \frac{\p P}{\p Z_2}=\frac{1}{2\p_P\Lambda(P,\w B_{0},\w A_{0})},\\
   &\p_P\Lambda(P,\w B_{0},\w A_{0})=\frac{\sqrt{2\w B_{0}-\frac{\gamma(\gamma+1)}
  {\gamma-1}
  \w A_{0}^{\frac{1}{\gamma-1}}P^{\frac{\gamma-1}{\gamma}}}}
  {2\gamma^{\frac12} \w A_{0}^{-\frac{1}{2\gamma}}\left(\w B_{0}-\frac{\gamma}{\gamma-1}
 \w  A_{0}^{\frac{1}{\gamma}}P^{1-\frac{1}{\gamma}}\right)
  P^{\frac{\gamma+1}{2\gamma}}}
  >0.\\
  \end{aligned}
    \end{equation}
    Thus  the implicit function theorem implies that  the equation \eqref{3-26} has a unique solution $ P=P( \bm Z,\w B_{0},\w A_{0}) $ and $W=W(\bm Z)$.

    \par Next, we consider the following nonlinear boundary value problem $ ({\mathbf{P}})$:
    \begin{equation}\label{3-28}
   ({\mathbf{P}}) \quad \begin{cases}
    \p_{y_1} \bm Z^-+diag(\lambda_1^-,\lambda_2^-)\p_{y_2} \bm Z^-=\bm K^-,
    &\quad {\rm in} \quad\mn^-,\\
    \p_{y_1} \bm Z^++diag(\lambda_1^+,\lambda_2^+)\p_{y_2} \bm Z^+=\bm K^+,
    &\quad {\rm in} \quad\mn^+,\\
     (Z_1^-,Z_2^-)=(Z_{10}^-,Z_{20}^-),&\quad {\rm on} \quad \Sigma_{0}^-,\\
     (Z_1^+,Z_2^+)=(Z_{10}^+,Z_{20}^+),&\quad {\rm on} \quad \Sigma_{0}^+,\\
      Z_1^-+ Z_2^-=2\arctan{(g^-)^{\prime}},&\quad {\rm on} \quad \Sigma_w^-,\\
      Z_1^++ Z_2^+=2\arctan{(g^+)^{\prime}},&\quad {\rm on} \quad \Sigma_w^+,\\
      Z_1^-+ Z_2^-=Z_1^++ Z_2^+ &\quad {\rm on} \quad \Sigma,\\
      P(\bm Z^-,\w B_{0}^-,\w A_{0}^-)= P(\bm Z^+,\w B_{0}^+,\w A_{0}^+)&\quad {\rm on} \quad \Sigma,\\
      \end{cases}
    \end{equation}
    where
    \begin{equation*}
    \begin{aligned}
     & \bm K^\pm=(K_1^\pm, K_2^\pm)^\top( \bm Z^\pm,\w B_{0}^\pm,\w A_{0}^\pm),\\
     &K_1^\pm( \bm Z^\pm,\w B_{0}^\pm,\w A_{0}^\pm)
  =-\frac{u_1^\pm((u_1^\pm)^2+(u_2^\pm)^2-(c^\pm)^2)}
  {((u_1^\pm)^2+(u_2^\pm)^2)((u_1^\pm)^2-(c^\pm)^2)}   +\frac{ c^\pm u_2^\pm
  \sqrt{(u_1^\pm)^2+(u_2^\pm)^2-(c^\pm)^2}}
 {((u_1^\pm)^2+(u_2^\pm)^2)((u_1^\pm)^2-(c^\pm)^2)},\\
   &K_2^\pm( \bm Z^\pm,\w B_{0}^\pm,\w A_{0}^\pm)
  =-\frac{u_1^\pm((u_1^\pm)^2+(u_2^\pm)^2-(c^\pm)^2)}
  {((u_1^\pm)^2+(u_2^\pm)^2)((u_1^\pm)^2-(c^\pm)^2)}   -\frac{ c^\pm u_2^\pm
  \sqrt{(u_1^\pm)^2+(u_2^\pm)^2-(c^\pm)^2}}
 {((u_1^\pm)^2+(u_2^\pm)^2)((u_1^\pm)^2-(c^\pm)^2)}.\\
 \end{aligned}
     \end{equation*}

\par Thus $\mathbf{ Problem \ 3.1} $ is reformulated as the following problem:
    \begin{problem}
    Given supersonic incoming flows $\bm Z_{0}^\pm $ at the entrance, find a unique piecewise smooth supersonic solution  $(\bm Z^{-},\bm Z^{+}) $ of the nonlinear boundary value problem  $ ({\mathbf{P}}) $.
    \end{problem}
     For $\mathbf{Problem \ 3.4}$,  we have the following theorem:
    \begin{theorem}
     Let
  \begin{equation*}
  \bm Z_{b}^{\pm}=( Z_{1b}^{\pm}, Z_{2b}^\pm)^\top,
    \end{equation*}
   where
   \begin{equation*}
   \begin{aligned}
   & Z_{1b}^\pm=-\Lambda(  \w P_b^\pm ,\w B_b^\pm, A_b^\pm),\quad
    Z_{2b}^\pm=\Lambda(  \w P_b^\pm,\w B_b^\pm, A_b^\pm),\\
  &\Lambda(\w P_b^\pm, \w B_b^\pm, A_b^\pm)=\int^{\w P_b^\pm}\frac{\sqrt{2 \w B_b^\pm-\frac{\gamma(\gamma+1)}{\gamma-1}
  ( A_b^\pm)^{\frac{1}{\gamma-1}}s^{\frac{\gamma-1}{\gamma}}}}
  {2\gamma^{\frac12} ( A_b^\pm)^{-\frac{1}{2\gamma}}\left( \w B_b^\pm-\frac{\gamma}{\gamma-1}
  (  A_b^\pm)^{\frac{1}{\gamma}}s^{1-\frac{1}{\gamma}}\right)
  s^{\frac{\gamma+1}{2\gamma}}}\de s,\\
 &\w B_b^\pm=\frac{1}{2}(\w u_{b}^\pm)^2+\frac{\w P_b^\pm}{(\gamma-1)\w \rho_b^\pm}.
  \end{aligned}
   \end{equation*}
Then given  functions $\bm Z_{0}^\pm(y_2) $, we define
\begin{equation}\label{3-29}
 \bm V_0^\pm=(\bm Z_{0}^\pm,\w B_{0}^\pm,\w A_{0}^\pm),\quad  \bm V_b^\pm=(\bm Z_{b}^\pm,\w B_{b}^\pm,A_{b}^\pm),
\end{equation}
and
 \begin{equation}\label{3-30}
  \begin{aligned}
\sigma(\bm V_0^-,\bm V_0^+,g^-,g^+)
&=
\| \bm V_{0}^--\bm V_{b}^-\|_{1,\alpha;[-m^-,0]}+\| \bm V_{0}^+-\bm V_{b}^+\|_{1,\alpha;[0,m^+]}\\
&\quad+\|g^-+1\|_{2,\alpha;[0,L]}
+\|g^+-1\|_{2,\alpha;[0,L]}.
\end{aligned}
\end{equation}
 There exist a small positive constant $ \sigma_3 $ depending only on  the data  such that  if
 \begin{equation}\label{3-31}
\sigma(\bm V_0^-,\bm V_0^+,g^-,g^+)\leq \sigma_3,
\end{equation}
    $\mathbf{ Problem \ 3.4} $ has  a unique smooth supersonic solution $ (\bm Z^{-},\bm Z^{+})\in C^{1,\alpha}(\overline{\mn^-})\times C^{1,\alpha}((\overline{\mn^+})$ satisfying
 \begin{equation}\label{3-32}
 \|\bm Z^--{\bm Z}_{b}^{-}\|_{1,\alpha;\overline{\mn^-}}
+\|\bm Z^+-{\bm Z}_{b}^{+}\|_{1,\alpha;\overline{\mn^+}}
\leq \mc_3\sigma(\bm V_0^-,\bm V_0^+,g^-,g^+),
\end{equation}
where $ \mc_3>0 $ depends only on  the data.
\end{theorem}
\subsection{Linearization}\noindent
\par In order to solve nonlinear boundary value problem $ ({\mathbf{P}})$,    we first linearize \eqref{3-28} and then solve
the linear equation in the  domain $ \mn $.
\par Define
\begin{equation}\label{3-33}
\begin{aligned}
\hat  {\bm Z}^\pm=\bm Z^\pm-\bm Z_{b}^\pm, \quad
\hat B^\pm=\w B_{0}^\pm- \w B_{b}^\pm, \quad \hat A^\pm=\w A_{0}^\pm-  A_{b}^\pm.
\end{aligned}
\end{equation}
Note that the background solutions satisfy
\begin{equation}\label{3-34}
\lambda_{2b}^\pm\p_{y_2}Z_{2b}^\pm=-\frac{1}{\w u_b^\pm}, \quad \lambda_{1b}^\pm\p_{y_2}Z_{1b}^\pm=-\frac{1}{\w u_b^\pm},
\end{equation}
where
\begin{equation*}
\lambda_{2b}^\pm=-\lambda_{1b}^\pm=\frac{\w\rho_b^\pm \w u_{b}^\pm \w c_b^\pm}{
\sqrt{(\w u_{b}^\pm)^2-(\w c_b^\pm)^2}}>0.
\end{equation*}
Then $ (\hat {\bm Z}^-,\hat {\bm Z}^+) $ satisfies
\begin{equation}\label{3-35}
  ({\mathbf{P}} )\quad  \begin{cases}
    \p_{y_1} \hat{ \bm Z}^-+diag(\lambda_1^-,\lambda_2^-)\p_{y_2} \hat{ \bm Z}^-=  {\bm H}^-,
    &\quad {\rm in} \quad\mn^-,\\
    \p_{y_1} \hat{ \bm Z}^++diag(\lambda_1^+,\lambda_2^+)\p_{y_2} \hat{ \bm Z}^+={\bm H}^+,
    &\quad {\rm in} \quad\mn^+,\\
    \hat{ \bm Z}^-= \hat{ \bm Z}_{0}^-,&\quad {\rm on} \quad \Sigma_{0}^-,\\
      \hat{ \bm Z}^+= \hat{ \bm Z}_{0}^+,&\quad {\rm on} \quad \Sigma_{0}^+,\\
      \hat  Z_1^-+ \hat  Z_2^-=2\arctan{(g^-)^{\prime}},&\quad {\rm on} \quad \Sigma_w^-,\\
       \hat Z_1^++ \hat  Z_2^+=2\arctan{(g^+)^{\prime}},&\quad {\rm on} \quad \Sigma_w^+,\\
       \hat Z_1^-+ \hat  Z_2^-=\hat  Z_1^++ \hat Z_2^+, &\quad {\rm on} \quad \Sigma,\\
      \alpha\hat Z_2^-+\beta \hat Z_1^+ = \alpha  \hat Z_1^-+\beta  \hat Z_2^++\mq, &\quad {\rm on} \quad \Sigma,\\
      \end{cases}
    \end{equation}
     where
    \begin{equation*}
    \begin{aligned}
    &\lambda_i^\pm=\lambda_i^\pm({ \bm Z}^\pm,\w B_{0}^\pm,\w A_{0}^\pm),i=1,2,\\
    & {\bm H}^\pm=( H_1^\pm,H_2^\pm)^\top( \hat{ \bm Z}^\pm,\hat B_{0}^\pm,\hat A_{0}^\pm) ,\\
    & H_1^\pm( \hat{ \bm Z}^\pm,\hat B_{0}^\pm,\hat A_{0}^\pm)
  =K_1^\pm+\frac{\lambda_1^\pm}{\lambda_{1b}^\pm\w u_b^\pm},\quad
   H_2^\pm( \hat{ \bm Z}^\pm,\hat B_{0}^\pm,\hat A_{0}^\pm)
  =K_2^\pm+\frac{\lambda_2^\pm}{\lambda_{2b}^\pm\w u_b^\pm},\\
    &\hat {\bm Z}_{0}^\pm=(\hat Z_{10}^\pm,\hat Z_{20}^\pm)^\top,\quad \hat Z_{i0}^\pm=Z_{i0}^\pm- Z_{ib}^\pm, \quad i=1,2, \\
    &\alpha=\alpha(y_1)=\frac{1}{2\int_0^1\p_{s}\Lambda(
    \w P_b^-+s(P^--\w P_b^-),\w B_{0}^-,\w A_{0}^-)\de s}(y_1,0),\\
  & \beta=\beta(y_1)=\frac{1}{2\int_0^1\p_{s}\Lambda(
    \w P_b^++s(P^+-\w P_b^+),\w B_{0}^+,\w A_{0}^+)\de s}(y_1,0),\\
   &  \mq=\left(P(\bm Z_{b}^-,\w B_{b}^-, A_{b}^-)
    -P(\bm  Z_{b}^-,\w B_{0}^-,\w A_{0}^-)\right.\\
    &\quad\quad\quad\left.+P(\bm Z_{b}^+,\w B_{0}^+,\w A_{0}^+)
    -P(\bm Z_{b}^+,\w B_{b}^+, A_{b}^+)\right)(0).\\
    \end{aligned}
    \end{equation*}
\par
  Define an iteration set as
    \begin{equation}\label{3-36}
    \begin{aligned}
    \mj_\delta:&=\left\{( \hat{\bm Z}^-, \hat{\bm Z}^+):   \|\hat{\bm Z}^-\|_{1,\alpha;\overline{\mn^-}}
+\| \hat{\bm Z}^+\|_{1,\alpha;\overline{\mn^+}}\leq \delta\right\}
\end{aligned}
\end{equation}
  with  a positive constant $ \delta$    to be determined later.  Given
 $  {\bm  Z}^\pm=\hat {\bm  Z}^\pm+{\bm  Z}_b^\pm $, where  $(\hat {\bm  Z}^-,\hat   {\bm  Z}^+)\in \mj_\delta$, find $(  \e{\bm  Z}^-,  \e{\bm  Z}^+)$ by solving  the following linear boundary value problem  $(\mathbf{LP})$:
  \begin{equation}\label{3-37}
  ({\mathbf{LP}} )\quad  \begin{cases}
    \p_{y_1}  \Z^-+diag(\lambda_1^-,\lambda_2^-)\p_{y_2}
    \Z^--\mathbb{D}^- \Z^-
    ={ \bm F}^-,
    &\quad {\rm in} \quad\mn^-,\\
    \p_{y_1} \Z^++diag(\lambda_1^+,\lambda_2^+)\p_{y_2}
    \Z^+-\mathbb{D}^+ \Z^+={ \bm F}^+,
    &\quad {\rm in} \quad\mn^+,\\
     \Z^-= \hat{\bm Z}_{0}^-,&\quad {\rm on} \quad \Sigma_{0}^-,\\
      \Z^+= \hat{\bm Z}_{0}^+,&\quad {\rm on} \quad \Sigma_{0}^+,\\
       \e Z_1^-+  \e Z_2^-=2\arctan{(g^-)^{\prime}},&\quad {\rm on} \quad \Sigma_w^-,\\
      \e Z_1^++  \e Z_2^+=2\arctan{(g^+)^{\prime}},&\quad {\rm on} \quad \Sigma_w^+,\\
      \e Z_1^-+  \e Z_2^-= \e Z_1^++  \e Z_2^+, &\quad {\rm on} \quad \Sigma,\\
      \alpha  \e Z_2^-+\beta \e Z_1^+ = \alpha \e  Z_1^-+\beta \e Z_2^++\mq, &\quad {\rm on} \quad \Sigma,\\
      \end{cases}
    \end{equation}
    where
    \begin{equation*}
    \begin{aligned}
    &\lambda_i^\pm=\lambda_i^\pm({ \bm Z}^\pm,\w B_{0}^\pm,\w A_{0}^\pm),i=1,2,\\
     &\mathbb{D}^\pm=\left(\begin{array}{ccc}
 d_{11}^\pm& d_{12}^\pm\\
d_{21}^\pm &  d_{22}^\pm\\
\end{array}\right)(\bm Z_b^\pm,
    \w B_b^\pm,A_b^\pm),\\
&d_{11}^\pm(y_2)=\p_{ Z_1^\pm}  H_1^\pm(\bm Z_b^\pm,
    \w B_b^\pm,A_b^\pm),\quad  d_{12}^\pm(y_2)=\p_{Z_2^\pm}  H_1^\pm(\bm Z_b^\pm,
    \w B_b^\pm, A_b^\pm),\\
     &d_{21}^\pm(y_2)=\p_{ Z_1^\pm} H_2^\pm(\bm Z_b^\pm,
    \w B_b^\pm,A_b^\pm),\quad  d_{22}^\pm(y_2)=\p_{Z_2^\pm}  H_2^\pm(\bm Z_b^\pm,
    \w B_b^\pm,A_b^\pm),\\
   & { \bm F}^\pm=\left(\begin{array}{ccc}
 F_1^\pm \\
 F_2^\pm\\
\end{array}\right) ( \hat{ \bm Z}^\pm,\hat B_{0}^\pm,\hat A_{0}^\pm),\\
     & F_1^\pm( \hat{ \bm Z}^\pm,\hat B_{0}^\pm,\hat A_{0}^\pm)= H_1^\pm-d_{11}^\pm\hat  Z_1^\pm-d_{12}^\pm\hat  Z_2^\pm,\\
      & F_2^\pm( \hat{ \bm Z}^\pm,\hat B_{0}^\pm,\hat A_{0}^\pm)= H_2^\pm-d_{21}^\pm\hat  Z_1^\pm-d_{22}^\pm\hat  Z_2^\pm.\\
     \end{aligned}
     \end{equation*}
   \par Use the abbreviation
    \begin{equation*}
    \sigma_{cd}=\sigma(\bm V_0^-,\bm V_0^+,g^-,g^+),
    \end{equation*}
  where   $ \sigma(\bm V_0^-,\bm V_0^+,g^-,g^+)$  is defined in \eqref{3-30}.
     Then   a direct computation yields
      \begin{equation}\label{3-38}
    \begin{aligned}
      &\sum_{j=1}^2\sum_{i=1}^2\sum_{I=\pm}\|d_{ji}^I\|
      _{1,\alpha;\overline{\mn^I}}
      \leq C,\\
     &\sum_{i=1}^2\sum_{I=\pm}\| F^I_i\|_{1,\alpha;\overline{\mn^I}}
     \leq C(\delta^2+ \sigma_{cd}),
\end{aligned}
     \end{equation}
     where $ C>0$ depends only on the data.
     \par Next, we establish the estimate of the eigenvalues.
     \begin{lemma}
      For $  (\hat{ \bm Z}^-,\hat{ \bm Z}^+ )\in \mj_\delta$, there  exist positive constants $\sigma_4>0$ and $ \mc_4>0 $   depending only on the data  such that  for $  \sigma_{cd}\leq\delta \leq\sigma_4$ with $ \sigma_4 $ sufficiently small, one has
       \begin{equation}\label{3-39}
\sum_{i=1}^2\sum_{I=\pm}\|\lambda_i^I\|
_{1,\alpha;\overline{\mn^I}}
\leq \mc_4.
\end{equation}
\end{lemma}
\begin{proof}
    Without loss of generality, we only consider the  estimate of  eigenvalues in $ \mn^+ $.  First, a simple calculation shows

\begin{equation}\label{3-40}
\begin{aligned}
&\sum_{i=1}^2\|\lambda_i^+( { \bm Z}^+,\w B_{0}^+,\w A_{0}^+)-\lambda_{ib}^+(\bm Z_b^+,
    \w B_b^+, A_b^+)\|
_{1,\alpha;\overline{\mn^+}}\\
&\leq C^+
\left(
\|\hat{  \bm Z}^{+}\|_{1,\alpha;\overline{\mn^+}}
+\| \hat  B^+\|_{1,\alpha; [0,m^+]}+
      \| \hat  A^+\|_{1,\alpha; [0,m^+]}\right)\\
&\leq C^+
(\delta+ \sigma_{cd}),
 \end{aligned}
\end{equation}
where $ C^+>0 $ depends only the data. Furthermore, there  exists a positive constant $ \Lambda^+ $    such  that
$\sum_{i=1}^2\|\lambda_{ib}^+\|_{1,\alpha;\overline{\mn^+}}\leq \Lambda^+$.
   Let $\sigma_4^+=\frac{\Lambda^+}{4C^+(C^++1)} $. Then if  $ \sigma_{cd}\leq\delta \leq\sigma_4^+ $, one has
   \begin{equation}\label{3-41}
\sum_{i=1}^2\|\lambda_i^+\|
_{1,\alpha;\overline{\mn^+}}
\leq 2\Lambda^+.
\end{equation}
 Similarly, there exist  positive constants $   \Lambda^- $  and $\sigma_4^-$   such that for $ \sigma_{cd}\leq \delta\leq \sigma_4^- $, one gets
   \begin{equation}\label{3-42}
\sum_{i=1}^2\|\lambda_i^-\|
_{1,\alpha;\overline{\mn^-}}
\leq 2\Lambda^-.
\end{equation}
Let $ \sigma_4=\min\{\sigma_4^+,\sigma_4^-\} $. We complete the proof of this Lemma.
    \end{proof}
      \section{The solution to the linear boundary value problem  $(\mathbf{LP})$ }\noindent
    \par In this section, we prove the following theorem.
\begin{theorem}
     For given $ ( \hat{ \bm Z}^-, \hat{ \bm Z}^+)\in \mj_\delta$, the problem $({\mathbf{LP}} )$ has a unique solution  $( \Z^-,  \Z^+)$ satisfying
    \begin{equation}\label{4-1}
\|\Z^-\|_{1,\alpha;\overline{\mn^-}}
+\| \Z^+\|_{1,\alpha;\overline{\mn^+}}
\leq \mc_\flat(\delta^2+ \sigma_{cd}),
\end{equation}
where   $\mc_\flat>0$ depends only on the data.
\end{theorem}
\par We   divide the domain $\mn $ into several parts  to prove Theorem 4.1. Let $ \zeta_{2}^+(y_1)=\psi_2^+(y_1;0,0) $  be the characteristic curve corresponding to $ \lambda_2^+ $ starting from point $(0, 0)$ and  $\zeta_{1}^+(y_1)=\psi_1^+(y_1;0,m^+) $  be the characteristic curve corresponding to $ \lambda_1^+ $ starting from point $(0, m^+)$, i.e,
    \begin{equation*}
    \left\{
 \begin{aligned}
 &\frac{\de \psi_2^+}{\de y_1}=\lambda_2^+(y_1,\psi_2^+(y_1;0,0)),\\
 &\psi_2^+(0;0,0)=0,
 \end{aligned}
 \right. \quad {\rm{and}} \quad
 \left\{
 \begin{aligned}
 &\frac{\de \psi_1^+}{\de y_1}=\lambda_1^+(y_1,\psi_1^+(y_1;0,m^+)),\\
 &\psi_1^+(0;0,m^+)=m^+.
 \end{aligned}
 \right.
    \end{equation*}
    Thus  one gets
\begin{equation}\label{4-2}
\zeta_{2}^+(y_1)=\int_{0}^{y_1}\lambda_2^+(s,\psi_2^+(s;0,0))\de s,\  {\rm{and}} \quad
\zeta_{1}^+(y_1)=m^++\int_{0}^{y_1}\lambda_1^+(s,\psi_1^+
(s;0,m^+))\de s.
\end{equation}
Let $ \zeta_{2}^-(y_1)=\psi_2^-(y_1;0,-m^-) $  be the characteristic curve corresponding to $ \lambda_2^- $ starting from point $(0, -m^-)$ and  $\zeta_{1}^-(y_1)=\psi_1^-(y_1;0,0) $  be the characteristic curve corresponding to $ \lambda_1^- $ starting from point $(0, 0)$, i.e,
    \begin{equation*}
    \left\{
 \begin{aligned}
 &\frac{\de \psi_2^-}{\de y_1}=\lambda_2^-(y_1,\psi_2^-(y_1;0,-m^-)),\\
 &\psi_2^+(0;0,-m^-)=-m^-,
 \end{aligned}
 \right. \quad {\rm{and}} \quad
 \left\{
 \begin{aligned}
 &\frac{\de \psi_1^-}{\de y_1}=\lambda_1^-(y_1,\psi_1^-(y_1;0,0)),\\
 &\psi_1^-(0;0,0)=0.
 \end{aligned}
 \right.
    \end{equation*}
    Hence  it holds that
\begin{equation}\label{4-3}
\zeta_{2}^-(y_1)=-m^-+\int_{0}^{y_1}\lambda_2^-(s,\psi_2^-
(s;0,-m^-))\de s,\ {\rm{and}} \quad
\zeta_{1}^-(y_1)=\int_{0}^{y_1}\lambda_1^-(s,\psi_1^-
(s;0,0))\de s.
\end{equation}
\par As shown in Fig 1,  the sub-domains are determined as follows. Let $ \mn_1$ be the union of two triangles $ \mn_1^- $ and $ \mn_1^+ $ which are bounded by the entrance $ \Sigma_{0}^- $ (or $ \Sigma_{0}^+ $), $\zeta_{1}^-$ (or $\zeta_{1}^+$),  and $\zeta_{2}^-$ (or $\zeta_{2}^+)$. Let $ \mn_2$ be the union of two triangles $ \mn_2^- $ and $ \mn_2^+ $ bounded by the  walls $ y_2=-m^- $ (or $ y_2=m^+ $), $\zeta_{1}^-$ (or $\zeta_{1}^+$), and $\zeta_2^-$ (or $\zeta_{2}^+$). Let $ \mn_3 $ be the diamond bounded by $\zeta_{1}^\pm$ and $\zeta_{2}^\pm$. Let $\mn_4 $ be the union of two diamonds bounded by $\zeta_{1}^\pm$, $\zeta_{2}^\pm$, the characteristic corresponding to $ \lambda_1^+ $ starting from the intersection point of $\zeta_{2}^+$ and $ y_2=m^+ $, and the characteristic corresponding to $ \lambda_2^- $ starting from the intersection point of $\zeta_{1}^-$ and $ y_2=-m^- $.
\begin{figure}
  \centering
  \includegraphics[width=12cm,height=4cm]{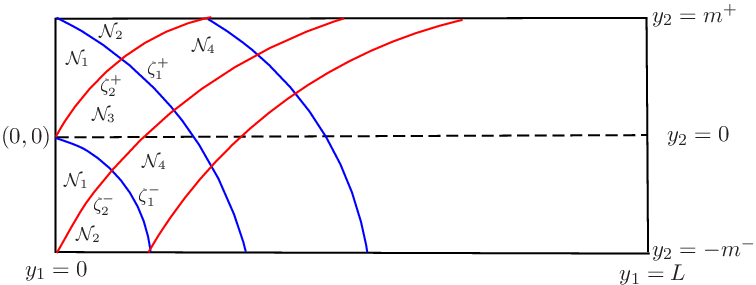}
  \caption{the solution of  $({\mathbf{LP}} )$}
\end{figure}
\subsection{The solution to the boundary value problem  $({\mathbf{LP}} )$ in $ \mn_1$}\noindent
\par In this subsection, we investigate  the problem $({\mathbf{LP}} )$ in $ \mn_1 $. Note that $ \mn_1$  is bounded by the entrance and the  characteristics curves. Hence the problem   $({\mathbf{LP}} )$ in $ \mn_1 $ can be regarded as the following initial value problem $({\mathbf{LP}})_1$:
 \begin{equation}\label{4-4}
 ({\mathbf{LP}})_1\ \ \begin{cases}
   \p_{y_1}  \Z^-+diag(\lambda_1^-,\lambda_2^-)\p_{y_2}
   \Z^--\mathbb{D}^- \Z^-
    ={ \bm F}^-,
    &\quad {\rm in} \quad\mn^-,\\
    \p_{y_1} \Z^++diag(\lambda_1^+,\lambda_2^+)\p_{y_2}
    \Z^+-\mathbb{D}^+ \Z^+={ \bm F}^+,
    &\quad {\rm in} \quad\mn^+,\\
      \Z^-=\hat{ \bm Z}_{0}^-,&\quad {\rm on} \quad \Sigma_{0}^-,\\
      \Z^+=\hat{ \bm Z}_{0}^+,&\quad {\rm on} \quad \Sigma_{0}^+.
      \end{cases}
    \end{equation}
    \par   Let $ (\xi_\ast^+,\eta_\ast^+)  $ be the  intersection point of the characteristic curves $ \zeta_{1}^+$ and $\zeta_{2}^+$, and let $ (\xi_\ast^-,\eta_\ast^-)  $ be the intersection point of the characteristic curves $\zeta_{1}^-$ and  $ \zeta_{2}^-$.
     Then $ \mn_1^\pm $ can be described as
     \begin{equation*}
     \begin{cases}
     \begin{aligned}
     &\mn_1^+=\{(y_1,y_2):0\leq y_1\leq \xi_\ast^+, \zeta_{2}^+(y_1)\leq y_2\leq  \zeta_{1}^+(y_1)\},\\
     &\mn_1^-=\{(y_1,y_2):0\leq y_1\leq \xi_\ast^-, \zeta_{2}^-(y_1)\leq y_2 \leq  \zeta_{1}^-(y_1)\}.
     \end{aligned}
     \end{cases}
     \end{equation*}
       Given any point $ (y_1,y_2)\in \mn_1$,  let $ \varphi^\pm_{i}(s;y_1,y_2) $ $(i=1,2) $ be the characteristic curves corresponding to $ \lambda^\pm_i $ starting from point $(y_1,y_2)$, i.e.,
    \begin{equation}\label{4-5}
 \begin{cases}
 \frac{\de \varphi^\pm_{i}}{\de s}
 =\lambda^\pm_{i}(s,\varphi^\pm_{i}(s;y_1,y_2)), \quad {\rm{for}} \quad 0\leq s\leq y_1,\\
 \varphi^\pm_{i}(y_1;y_1,y_2)=y_2.
 \end{cases}
    \end{equation}
    \par  For any point $ (y_1,y_2)$ in $\mn_1^+ $ (or  in $\mn_1^- $),  it follows from \eqref{4-5}  that the characteristic associated with $ \lambda_i^+ $ (or $ \lambda_i^-$) starting from $ (y_1,y_2) $  has a unique intersection point  $(0, \eta_{i}^+)$ (or $(0, \eta_{i}^-)$)  with  the entrance (See Fig 2 and Fig 3). Hence $ \eta_{i}^\pm $  can be regarded as functions of $ (y_1,y_2)$.
     \begin{figure}

\begin{minipage}{0.44\textwidth}
  \centering
  \includegraphics[width=6.5cm,height=3cm]{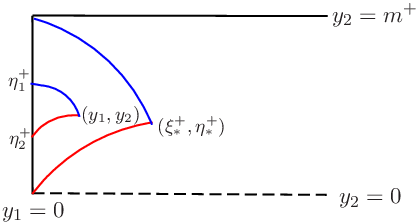}
  \caption{the  solution of  $({\mathbf{LP}})_1$ in $\mn_1^+$ }
  \end{minipage}
 \begin{minipage}{0.44\textwidth}
  \centering
  \includegraphics[width=6.5cm,height=3cm]{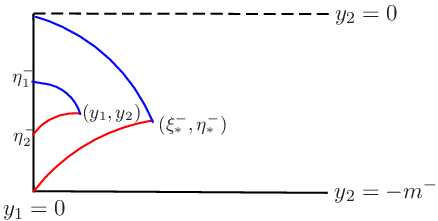}
  \caption{the  solution of  $({\mathbf{LP}})_1$ in $\mn_1^-$}
   \end{minipage}
   \end{figure}
    For convenience, we only consider the  problem $({\mathbf{LP}})_1$ in $ \mn^+$.
    The  problem $({\mathbf{LP}})_1$ in $ \mn^-$ can be treated in the same way.
    \par  Next, we consider the following problem:
    \begin{equation}\label{4-6}
    \begin{cases}
\p_{y_1} \Z^++diag(\lambda_1^+,\lambda_2^+)\p_{y_2}
    \Z^+-\mathbb{D}^+ \Z^+={ \bm F}^+,&\quad {\rm on} \quad \mn^+\\
   \Z^+=\hat{ \bm Z}_{0}^+,&\quad {\rm on} \quad \Sigma_{0}^+.
      \end{cases}
    \end{equation}
    For any given point $(y_1,y_2)\in \mn_1^+$, along the characteristic curve $\varphi_{i}^+(s;y_1,y_2) $, there holds
     \begin{equation}\label{4-7}
     \begin{cases}
     \begin{aligned}
    & \e Z_1^+(y_1,y_2)
    =  \hat Z_{10}^+(\eta_{1}^+)
    +\int_{0}^{y_1}\left(\sum_{i=1}^2d_{1i}^+\e Z_i^++ F_1^+\right)(s,\varphi_{1}^+(s; y_1,y_2))\de s,\\
    & \ \e Z_2^+(y_1,y_2)=\hat  Z_{20}^+(\eta_{2}^+)
    +\int_{0}^{y_1}\left(\sum_{i=1}^2d_{2i}^+\e Z_i^++  F_2^+\right)(s,\varphi_{2}^+(s;y_1,y_2))\de s.
    \end{aligned}
    \end{cases}
    \end{equation}
   \par By Picard iteration in \cite{LY85}, the integral equation \eqref{4-7} has a unique solution $ \Z^+\in C^{1}(\overline{\mn_1^+}) $. Thus the Cauchy problem \eqref{4-6}  has a unique solution $ \Z^+\in C^{1}(\overline{\mn_1^+}) $. In the following, we establish the estimate of the solution.
    \begin{proposition}
       For $ (\hat{ \bm Z}^-,\hat{ \bm Z}^+)\in \mj_\delta $,  there exist a   positive constant  $ \mc_{5}^+ $  depending only on the data  such that   the solution $  \Z^+$ to the problem \eqref{4-6} satisfies
    \begin{equation}\label{4-8}
    \| \Z^+\|_{1,\alpha;\overline {\mn_1^+}}
    \leq \mc_5^+\left(\| \hat{  \bm Z}_{0}^+\|_{1,\alpha; [0,m^+]}+\|{ \bm F}^+\|_{1,\alpha;\overline{\mn^+}}\right)
\leq \mc_5^+(\sigma_{cd}+\delta^2).
    \end{equation}
\end{proposition}
\begin{proof}
 For $ 0\leq s\leq  \xi_{\ast}^+ $, define $  \overline{\mn_1^+}(s)=\{0\leq w_1\leq s, \zeta_{2}^+(w_1)\leq w_2\leq  \zeta_{1}^+(w_1)\} $ and set
     \begin{equation*}
     \begin{cases}
     \begin{aligned}
     \mz^+(s)&=\sum_{i=1}^2\sup _{ \bm w=(w_1,w_2)\in   \overline{\mn_1^+}(s)}|  \e Z_i^+(\bm w)|,\\
     D \mz^+(s)&=\sum_{i=1}^2\sup _{\bm w=(w_1,w_2)\in   \overline{\mn_1^+}(s)}|  D\e Z_i^+(\bm w)|,\\
     [D \mz^+]_\alpha(s)&=\sum_{i=1}^2\sup _{\substack{\bm w=(w_1,w_2)\in   \overline{\mn_1^+}(s)\\\w {\bm w} =(\w w_1,\w w_2)\in  \overline{ \mn_1^+}(s)}}
     d^{-\alpha}(\bm w,\w {\bm w})|  D \e Z_i^+( {\bm w})- D \e Z_i^+(\w {\bm w})|,
     \end{aligned}
     \end{cases}
     \end{equation*}
    where
     \begin{equation*}
      d(\bm w,\w {\bm w})=|w_1-\w w_1|+|w_2-\w w_2|.
      \end{equation*}
      Obviously,
     \begin{equation}\label{4-9}
     \|
    \Z^+\|_{1,\alpha;\overline{\mn_1^+}}=\sup_{ 0\leq s\leq  \xi_{\ast}^+} \mz^+(s)+\sup_{0\leq s\leq \xi_{\ast}^+}
      D \mz^+(s)+\sup_{ 0\leq s\leq  \xi_{\ast}^+} [D \mz^+]_\alpha(s).
    \end{equation}
    \par   Next,  we divide the proof into two steps.
       \par $\textbf{Step 1} $. The estimate of $ \|
    \Z^+\|_{0,0;\overline{\mn_1^+}} $.    For $ 0\leq y_1\leq \xi_{\ast}^+ $,  it follows from \eqref{4-7} that
    \begin{equation}\label{4-10}
    \mz^+(y_1)
    \leq C_{11}\left(\|\hat{  \bm Z}_{0}^+\|_{0,0; [0,m^+]}+\|{ \bm F}^+\|_{0,0;\overline{\mn_1^+} }\right)+C_{11}\int_{0}^{y_1} \mz^+(s)\de s,\\
     \end{equation}
      where  $  C_{11}>0 $   depends only on the data. Applying the  Gronwall inequality yields
    \begin{equation}\label{4-11}
   \| \Z^+\|_{0,0;\overline{\mn_1^+}}
        \leq C_{11} e^{ C_{11}L}\left(\| \hat{  \bm Z}_{0}^+\|_{1,\alpha; [0,m^+]}+\|{  \bm F}^+\|_{1,\alpha;\overline{\mn^+}}\right).
\end{equation}
\par $\textbf{Step 2} $. The estimate of $ \|D
    \Z^+\|_{0,\alpha;\overline{\mn^+_1}} $. First,
    differentiating \eqref{4-6}  with respect to $ y_1 $ and $ y_2 $ to obtain that
      \begin{equation}\label{4-12}
      \begin{cases}
      \begin{aligned}
      &\p_{y_1}(\p_{y_j} \e Z_1^+)+\lambda_1^+\p_{y_2}(\p_{y_j} \e Z_1^+)-\sum_{i=1}^2d_{1i}^+\p_{y_j}\e Z_i^+\\
      &=\p_{y_j} F_1^+-\p_{y_j}\lambda_1^+\p_{y_2} \e Z_1^++\sum_{i=1}^2\p_{y_j}d_{1i}^+\e Z_i^+,\\
      &\p_{y_1}(\p_{y_j} \e Z_2^+)+\lambda_2^+\p_{y_2}(\p_j \e Z_2^+)-\sum_{i=1}^2d_{2i}^+\p_{y_j}\e Z_i^+\\
      &=\p_{y_j} F_2^+-\p_{y_j}\lambda_2^+\p_{y_2} \e Z_2^++\sum_{i=1}^2\p_{y_j}d_{2i}^+\e Z_i^+,\\
      \end{aligned}
       \end{cases}
       \end{equation}
       where $ \p_{y_1}d_{ij}^+=0, (i,j=1,2 ) $.
      Then along the characteristic curve $ \varphi_{i}^+(s;y_1,y_2) $, one gets
       \begin{equation}\label{4-13}
       \begin{cases}
       \begin{aligned}
        &\p_{y_j} \e Z_1^+(y_1,y_2)\\
       &=\p_{y_j} \e Z_1^+(0,\eta_1^+)
    +\int_{0}^{y_1}\sum_{i=1}^2d_{1i}^+\p_{y_j}\e Z_i^+(s,\varphi_{1}^+(s;y_1,y_2))\de s \\
     &\quad +\int_{0}^{y_1}\left(\p_{y_j} F_1^+-\p_{y_j}\lambda_1^+\p_{y_2} \e Z_1^++\sum_{i=1}^2\p_{y_j}d_{1i}^+\e Z_i^+\right)(s,\varphi_{1}^+(s;y_1,y_2))\de s,\\
 &\p_{y_j} \e Z_2^+(y_1,y_2)\\
     &=\p_{y_j} \e Z_2^+(0,\eta_2^+)
    +\int_{0}^{y_1}\sum_{i=1}^2d_{2i}^+\p_{y_j}\e Z_i^+(s,\varphi_{2}^+(s;y_1,y_2))\de s \\
    &\quad +\int_{0}^{y_1}\left(\p_{y_j} F_2^+-\p_{y_j}\lambda_2^+\p_{y_2} \e Z_2^++\sum_{i=1}^2\p_{y_j}d_{2i}^+\e Z_i^+\right)(s,\varphi_{2}^+(s;y_1,y_2))\de s.
     \end{aligned}
     \end{cases}
      \end{equation}
  Furthermore, by the boundary condition  in   $ \eqref{4-6}  $, one can derive that
      \begin{equation}\label{4-15}
      \begin{cases}
      \begin{aligned}
      \p_{y_2}\e Z_1^+(0,\eta_1^+)=(\hat{  Z}_{10}^+)^\prime(\eta_1^+),\quad
      \p_{y_1}\e Z_1^+(0,\eta_1^+)=\left(-\lambda_1^+(\hat{  Z}_{10}^+)^\prime+\sum_{i=1}^2d_{1i}^+\hat{  Z}_{i0}^++F_1^+\right)(0,\eta_1^+),\\
      \p_{y_2}\e Z_2^+(0,\eta_2^+)=(\hat{  Z}_{20}^+)^\prime(\eta_2^+),\quad
      \p_{y_1}\e Z_2^+(0,\eta_2^+)=\left(-\lambda_2^+(\hat{  Z}_{20}^+)^\prime+\sum_{i=1}^2d_{2i}^+\hat{  Z}_{i0}^++F_2^+\right)(0,\eta_2^+).
      \end{aligned}
       \end{cases}
       \end{equation}
       For $ 0\leq y_1\leq \xi_{\ast}^+ $, it follows from
       \eqref{4-13}-\eqref{4-15} that
     \begin{equation}\label{4-16}
     \begin{aligned}
     D \mz^+(y_1)
    \leq C_{12}\int_{0}^{y_1}D\mz^+(s)\de s
    +C_{12}\left(\| \hat{\bm Z}_{0}^+\|_{1,\alpha; [0,m^+]}+\|{  \bm F}^+\|_{1,\alpha;\overline{\mn^+}}
     \right),\\
      \end{aligned}
       \end{equation}
    where  $  C_{12}>0 $   depends only on the data.   By the  Gronwall inequality, there holds
    \begin{equation}\label{4-17}
   \|D \Z^+\|_{0,0;\overline{\mn^+_1}}
        \leq C_{12} e^{C_{12}L}\left(\| \hat{  \bm Z}_{0}^+\|_{1,\alpha; [0,m^+]}+\|{  \bm F}^+\|_{1,\alpha;\overline{\mn^+}}\right).
      \end{equation}
      \par Next, to obtain the  estimate of $ [D \Z^+]_{0,\alpha;\overline{\mn^+_1}} $, we first establish the following estimate:
      \begin{equation}\label{4-18}
 \sum_{i=1}^2\left(\|\varphi_{i}^+(s;\cdot,\cdot)\|_{1;\overline{\mn_1^+(s)}}
+\|\eta_{i}^+\|
_{1;\overline{\mn_1^+}}\right)
\leq C,
\end{equation}
where $ C>0 $ depends only on the data but is  independent of $ s $.
 \par     Differentiating  $\eqref{4-5} $ with respect to $ y_1 $ and $ y_2 $ to get
   \begin{equation*}
   \begin{cases}
    \begin{aligned}
 &\frac{\p \varphi_{i}^+(s;y_1,y_2)}{\p y_2}
 =e^{\int_{y_1}^{s}\p_{y_2}\lambda_i^+
 (\mu,\varphi_{i}^+(\mu;y_1,y_2))\de \mu},\\
 &\frac{\p \varphi_{i}^+(s;y_1,y_2)}{\p y_1}
 =-\lambda_i^+(y_1,y_2)e^{\int_{y_1}^{s}\p_{y_2}\lambda_i^+
 (\mu,\varphi_{i}^+(\mu;y_1,y_2))\de \mu}.\\
 \end{aligned}
 \end{cases}
      \end{equation*}
      Note that
        \begin{equation}\label{4-19}
  \eta_{i}^+=y_2-\int_{0}^{y_1}\lambda_i^+(s, \varphi_{i}^+(s;y_1,y_2))\de s.
  \end{equation}
  Then taking derivatives \eqref{4-19} with respect to $ y_1 $ and $ y_2 $ to obtain
    \begin{equation*}
    \begin{cases}
    \begin{aligned}
    &\frac{\p \eta_{i}^+}{\p y_2}=1-\int_{0}^{y_1}\p_{y_2}\lambda_i^+
    e^{\int_{y_1}^{s}\p_{y_2}\lambda_i^+(\mu,\varphi_{i}^+(\mu;y_1,y_2)) \de \mu}\de s,\\
     &\frac{\p \eta_{i}^+}{\p y_1}=\int_{0}^{y_1}\lambda_i^+\p_{y_2}\lambda_i^+
    e^{\int_{y_1}^{s}\p_{y_2}\lambda_i^+(\mu,\varphi_{i}^+(\mu;y_1,y_2)) \de \mu}\de s-\lambda_i^+(y_1,y_2).\\
    \end{aligned}
    \end{cases}
      \end{equation*}
       Thus one can infer that
 \begin{equation*}
 \sum_{i=1}^2\left(\|\varphi_{i}^+(s;\cdot,\cdot)\|_{1;\overline{\mn_1^+(s)}}
+\|\eta_{i}^+\|
_{1;\overline{\mn_1^+}}\right)
\leq C.
 \end{equation*}
 \par  For any two points $   \bm Q=(y_{1Q},y_{2Q}) $ and $ \bm P=(y_{1P},y_{2P})\in \mn^+_1 $,
       let $ (0,\eta_{i}^+(\bm\ell)) $ $( \ell= P,  Q) $  be the intersection point of the characteristic associated with $ \lambda_i^+ $ staring from point $(y_{1\ell},y_{2\ell})$ and the entrance $ \Sigma_{0}^+ $. Without loss of generality, we assume that $ y_{1P}\leq y_{1Q} $.
     \par Define
       \begin{equation*}
d(\bm P,\bm Q)=|y_{1P}-y_{1Q}|+|y_{2P}-y_{2Q}|.
    \end{equation*}
     Note that
  \begin{equation*}
    \begin{aligned}
    &d^{-\alpha}(\bm P,\bm Q)\left(|D \e Z_1^+(\bm P)-D \e Z_1^+(\bm Q)|+|D \e Z_2^+(\bm P)-D \e Z_2^+(\bm Q)|\right)\\
    &\quad\leq d^{-\alpha}(\bm P,\bm Q)\left(|\p_{y_1} \e Z_1^+(\bm P)-\p_{y_1} \e Z_1^+(\bm Q)|+|\p_{y_2} \e Z_1^+(\bm P)-\p_{y_2} \e Z_1^+(\bm Q)|\right)\\
    &\quad\quad+d^{-\alpha}(\bm P,\bm Q)\left(|\p_{y_1} \e Z_2^+(\bm P)-\p_{y_1} \e Z_2^+(\bm Q)|+|\p_{y_2} \e Z_2^+(\bm P)-\p_{y_2} \e Z_2^+(\bm Q)|\right)\\
    &=J_{11}^++J_{12}^++J_{21}^++J_{22}^+.
       \end{aligned}
   \end{equation*}
     For convenience, we only estimate $ J_{11}^+ $, the other terms can be estimated in the same way.
      The straightforward
computations give
      \begin{equation}\label{4-20}
    \begin{aligned}
   J_{11}^+&\leq d^{-\alpha}(\bm P, \bm Q)|\p_{y_1} \e Z_1^+(0,\eta_1^+(\bm  P))-\p_{y_1} \e Z_1^+(0,\eta_1^+(\bm  Q))|\\
   &\quad+d^{-\alpha}(\bm P, \bm Q)\int_{0}^{y_{1P}}\left|\left(\sum_{i=1}^2
    d_{1i}^+
   \p_{y_1}\e Z_i^++\p_{y_1}  F_1^+-\p_{y_1}\lambda_1^+\p_{y_2} \e Z_1^+\right)
    (s,\varphi_{1}^+(s;y_{1P},y_{2P}))\right.\\
    &\qquad\qquad\qquad\qquad\quad\left.-\left(\sum_{i=1}^2
    d_{1i}^+
   \p_{y_1}\e Z_i^++\p_{y_1}  F_1^+-\p_{y_1}\lambda_1^+\p_{y_2} \e Z_1^+\right)
     (s,\varphi_{1}^+(s;y_{1Q},y_{2Q}))\right|\de s\\
    &\quad+ d^{-\alpha}(\bm P,\bm Q)\int_{y_{1P}}^{y_{1Q}}\left|\left(\sum_{i=1}^2d_{1i}^+\p_{y_1}\e Z_i^+
     +\p_{y_1}  F_1^+-\p_{y_1}\lambda_1^+\p_{y_2}Z_1^+ \right)
    (s,\varphi_{1}^+(s;y_{1Q},y_{2Q}))\right|\de s\\
     & = J_{111}^++J_{112}^++J_{113}^+.
   \end{aligned}
   \end{equation}
 For the first term, one derives
   \begin{equation*}
    \begin{aligned}
    J_{111}^+&= d^{-\alpha}(\bm P, \bm Q)|\p_{y_1} \e Z_1^+(0,\eta_1^+(\bm  P))-\p_{y_1} \e Z_1^+(0,\eta_1^+(\bm  Q))|\\
   &=d^{-\alpha}(\bm P, \bm Q)\left|\left(-\lambda_1^+(\hat{  Z}_{10}^+)^\prime+\sum_{i=1}^2d_{1i}^+\hat{  Z}_{i0}^++F_1^+\right)(0,\eta_1^+(\bm  P))\right.\\
    &\qquad\qquad\qquad -\left.\left(-\lambda_1^+(\hat{  Z}_{10}^+)^\prime+\sum_{i=1}^2d_{1i}^+\hat{  Z}_{i0}^++F_1^+\right)(0,\eta_1^+(\bm  Q))\right|\\
   &\leq \frac{{|\eta_{1}^+(\bm P)-\eta_{1}^+(\bm Q)|^\alpha}}{d^{\alpha}(\bm P,\bm Q)}
    \frac{|(\lambda_1^+(\hat{  Z}_{10}^+)^\prime+ F_1^+)(0,\eta_{1}^+(\bm  P))-(\lambda_1^+(\hat{  Z}_{10}^++ F_1^+)^\prime)(0,\eta_{1}^+(\bm  Q))|}{|\eta_{1}^+(\bm P)-\eta_{1}^+(\bm Q)|^\alpha}\\
    &\quad+\frac{{|\eta_{1}^+(\bm P)-\eta_{1}^+(\bm Q)|^\alpha}}{d^{\alpha}(\bm P,\bm Q)}
    \frac{|\left(\sum_{i=1}^2d_{1i}^+\hat Z_{i0}^+\right)(0,\eta_1^+(\bm  P))-\left(\sum_{i=1}^2d_{1i}^+\hat Z_{i0}^+\right)(0,\eta_1^+(\bm  Q))|}{|\eta_{1}^+(\bm P)-\eta_{1}^+(\bm Q)|^\alpha}\\
&\leq C\left(\|{ \bm F}^+\|_{1,\alpha;\overline{\mn^+_1}}+ \| \hat  {\bm Z}_{0}^+\|_{1,\alpha; [0,m^+]}\right).\\
\end{aligned}
      \end{equation*}
      For the second term, one has
      \begin{equation*}
    \begin{aligned}
J_{112}^+&=d^{-\alpha}(\bm P, \bm Q)\int_{0}^{y_{1P}}\left|\left(\sum_{i=1}^2
    d_{1i}^+
   \p_{y_1}\e Z_i^++\p_{y_1}  F_1^+-\p_{y_1}\lambda_1^+\p_{y_2} \e Z_1^+\right)
    (s,\varphi_{1}^+(s;y_{1P},y_{2P}))\right.\\
    &\qquad\qquad\qquad\qquad\left.-\left(\sum_{i=1}^2
    d_{1i}^+
   \p_{y_1}\e Z_i^++\p_{y_1}  F_1^+-\p_{y_1}\lambda_1^+\p_{y_2} \e Z_1^+\right)
     (s,\varphi_{1}^+(s;y_{1Q},y_{2Q}))\right|\de s\\
     &\leq \int_{0}^{y_{1P}}
     \frac{\left|\sum_{i=1}^2
    d_{1i}^+
   \p_{y_1}\e Z_i^+
   (s,\varphi_{1}^+(s;y_{1P},y_{2P}))-\sum_{i=1}^2
    d_{1i}^+
   \p_{y_1}\e Z_i^+
     (s,\varphi_{1}^+(s;y_{1Q},y_{2Q}))\right|}
     {|\varphi_{1}^+(s;y_{1P},y_{2P})-\varphi_{1}^+(s;y_{1Q},y_{2Q})|}
     \times\\
&\qquad\qquad\frac{|\varphi_{1}^+(s;y_{1P},y_{2P})-\varphi_{1}^+(s;y_{1Q},y_{2Q})|
}{d^{\alpha}(\bm P,\bm Q)} \de s\\
&\quad+\int_{0}^{y_{1P}}
    \frac{\left|\p_{y_1} F_1^+ (s,\varphi_{1}^+(s;y_{1P},y_{2P}))-\p_{y_1} F_1^+(s,\varphi_{1}^+(s;y_{1Q},y_{2Q}))\right|} {|\varphi_{1}^+(s;y_{1P},y_{2P})-\varphi_{1}^+(s;y_{1Q},y_{2Q})|^\alpha}\times
    \\
    &\qquad\qquad\quad\frac {|\varphi_{1}^+(s;y_{1P},y_{2P})-\varphi_{1}^+(s;y_{1Q},y_{2Q})|^\alpha}
    {d^{\alpha}(\bm P,\bm Q)}
    \de s\\
    \end{aligned}
      \end{equation*}
       \begin{equation*}
    \begin{aligned}
    &\quad\quad+\int_{0}^{y_{1P}}\frac{\left|\p_{y_1}\lambda_1^+\p_{y_2} \e Z_1^+ (s,\varphi_{1}^+(s;y_{1P},y_{2P}))-\p_{y_1}\lambda_1^+\p_{y_2} \e Z_1^+(s,\varphi_{1}^+(s;y_{1Q},y_{2Q}))\right|} {|\varphi_{1}^+(s;y_{1P},y_{2P})-\varphi_{1}^+(s;y_{1Q},y_{2Q})|^\alpha}\times
    \\
    &\qquad\qquad\quad\frac {|\varphi_{1}^+(s;y_{1P},y_{2P})-\varphi_{1}^+(s;y_{1Q},y_{2Q})|^\alpha}
    {d^{\alpha}(\bm P,\bm Q)}
    \de s\\
     &\leq C\left(\|{ \bm F}^+\|_{1,\alpha;\overline{\mn^+_1}}+\| \Z^+\|_{1,0;\overline{\mn^+_1}}\right)\\
&\quad+C\int_{0}^{y_{1P}}\frac{\left|\p_{y_2} \e Z_1^+ (s,\varphi_{1}^+(s;y_{1P},y_{2P}))-\p_{y_2} \e Z_1^+(s,\varphi_{1}^+(s;y_{1Q},y_{2Q}))\right|} {|\varphi_{1}^+(s;y_{1P},y_{2P})-\varphi_{1}^+(s;y_{1Q},y_{2Q})|^\alpha}\de s\\
 &\quad+C\int_{0}^{y_{1P}}
     \frac{\sum_{i=1}^2\left|\p_{y_1}\e Z_i^+ (s,\varphi_{1}^+(s;y_{1P},y_{2P}))-\p_{y_1}\e Z_i^+ (s,\varphi_{1}^+(s;y_{1Q},y_{2Q}))\right|}
     {|\varphi_{1}^+(s;y_{1P},y_{2P})-\varphi_{1}^+(s;y_{1Q},y_{2Q})|^\alpha}
      \de s\\
 &\leq C\left(\|{ \bm F}^+\|_{1,\alpha;\overline{\mn^+_1}}+\| \Z^+\|_{1,0;\overline{\mn^+_1}}\right)
+C\int_{0}^{y_{1P}}[D\mz^+]_\alpha(s)\de s.\\
 \end{aligned}
      \end{equation*}
      For the last term, one obtains
      \begin{equation*}
    \begin{aligned}
J_{113}^+&= d^{-\alpha}(\bm P,\bm Q)\int_{y_{1P}}^{y_{1Q}}\left|\left(\sum_{i=1}^2d_{1i}^+\p_{y_1}\e Z_i^+
     +\p_{y_1}  F_1^+-\p_{y_1}\lambda_+^+\p_{y_2} \e Z_1^+\right)
    (s,\varphi_{1}^+(s;y_{1Q},y_{2Q}))\right|\de s\\
    &\leq  C\left(\|{ \bm F}^+\|_{1,\alpha;\overline{\mn^+_1}}+\| \Z^+\|_{1,0;\overline{\mn^+_1}}\right).\\
    \end{aligned}
   \end{equation*}
 Collecting the above estimates and \eqref{4-11} and \eqref{4-17} gives that
\begin{equation}\label{4-21}
    \begin{aligned}
    J_{11}^+\leq
      C\int_{0}^{y_{1P}}[D\mz^+]_\alpha(s)\de s
     +C\left(\|   \hat{\bm Z}_{0}^+\|_{1,\alpha; [0,m^+]}+\| { \bm F}^+\|_{1,\alpha;\overline{\mn^+_1}}\right).\\
      \end{aligned}
    \end{equation}
    Similarly, one has
    \begin{equation}\label{4-22}
    \begin{aligned}
   J_{12}^++J_{21}^++J_{22}^+\leq
      C\int_{0}^{y_{1P}}[D\mz^+]_\alpha(s)\de s
     +C\left(\|   \hat{\bm Z}_{0}^+\|_{1,\alpha; [0,m^+]}+\| { \bm F}^+\|_{1,\alpha;\overline{\mn^+_1}}\right).\\
     \end{aligned}
    \end{equation}
    For  $ 0\leq y_{1P}\leq y_{1Q}\leq y_1\leq \xi_{\ast}^+  $, it follows from \eqref{4-21} and \eqref{4-22} that
    \begin{equation}\label{4-23}
     [D \mz^+]_\alpha(y_1)
     \leq C_{13}\int_{0}^{y_1}[D\mz^+]_\alpha(s)\de s
    +C_{13}\left(\|   \hat{\bm Z}_{0}^+\|_{1,\alpha; [0,m^+]}+\|  {\bm F}^+\|_{1,\alpha;\overline{\mn^+}}
     \right),\\
     \end{equation}
    where  $  C_{13}>0 $   depends only on the data.  By the Gronwall inequality, there holds
    \begin{equation}\label{4-24}
   [ D\Z^+]_{0,\alpha;\overline{\mn_1^+}}
        \leq C_{13} e^{ C_{13}L}\left(\| \hat{  \bm Z}_{0}^+\|_{1,\alpha; [0,m^+]}+\| { \bm F}^+\|_{1,\alpha;\overline{\mn^+}}\right).
        \end{equation}
        Let $ \mc_5^+=C_{11} e^{ C_{11}L}+C_{12} e^{ C_{12}L}+C_{13} e^{ C_{13}L} $. Combining \eqref{4-11}, \eqref{4-17} and \eqref{4-26} yields that
        \begin{equation*}
    \| \Z^+\|_{1,\alpha;\overline{\mn_1^+}}
    \leq \mc_{5}^+\left(\|   \hat{\bm Z}_{0}^+\|_{1,\alpha; [0,m^+]}+\| {\bm F}^+\|_{1,\alpha;\overline{\mn^+}}\right)
\leq \mc_5^+(\sigma_{cd}+\delta^2).
    \end{equation*}
    \end{proof}
     \subsection{   The solution to the boundary value problem  $({\mathbf{LP}} )$ in $ \mn_1\cup\mn_2$}\noindent
  \par In this subsection, we investigate  the problem $({\mathbf{LP}} )$ with the boundary conditions on $ \Sigma_w^- $ and   $ \Sigma_w^+ $ respectively.  The problem $({\mathbf{LP}} )$  can be regarded as the following problem $({\mathbf{LP}})_2$:
   \begin{equation}\label{4-25}
  ({\mathbf{LP}} )_2\quad  \begin{cases}
    \p_{y_1}  \Z^-+diag(\lambda_1^-,\lambda_2^-)\p_{y_2}
    \Z^--\mathbb{D}^- \Z^-
    ={ \bm F}^-,
    &\quad {\rm in} \quad\mn^-,\\
    \p_{y_1} \Z^++diag(\lambda_1^+,\lambda_2^+)\p_{y_2}
    \Z^+-\mathbb{D}^+ \Z^+= {\bm F}^+,
    &\quad {\rm in} \quad\mn^+,\\
      \Z^-= \hat{\bm Z}_{0}^-,&\quad {\rm on} \quad \Sigma_{0}^-,\\
      \Z^+= \hat{\bm Z}_{0}^+,&\quad {\rm on} \quad \Sigma_{0}^+,\\
       \e Z_1^-+  \e Z_2^-=2\arctan{(g^-)^{\prime}},&\quad {\rm on} \quad \Sigma_w^-,\\
       \e Z_1^++  \e Z_2^+=2\arctan{(g^+)^{\prime}},&\quad {\rm on} \quad \Sigma_w^+.\\
     \end{cases}
    \end{equation}
    This problem can be solved  by similar arguments in subsection 4.1 using the reflection of the characteristic curves by the upper and lower walls.
      \par  Let $ \xi_{\ast\ast}^+  $ be the  intersection point of the characteristic curve $ \zeta_{2}^+ $  and the upper wall $ m^{+} $, and let $ \xi_{\ast\ast}^-  $ be the  intersection point of the characteristic curve $ \zeta_{1}^- $ and the  lower wall $ -m^{-} $, i.e,
      \begin{equation}\label{4-26}
      \int_{0}^{\xi_{\ast\ast}^+}
      \lambda_2^+(s,\psi_2^+(s;0,0))\de s=m^+, \ \
       \int_{0}^{\xi_{\ast\ast}^-}\lambda_1^-(s,\psi_1^-(s;0,0))\de s=-m^-.
       \end{equation}
       Then $ \mn_2^\pm $ can be described as
     \begin{equation*}
     \begin{cases}
     \begin{aligned}
     \mn_2^-&=\{0\leq y_1\leq \xi_{\ast}^-,
     -m^-\leq y_2\leq  \zeta_{2}^-(y_1)\}\cup
     \{\xi_{\ast}^-\leq y_1\leq \xi_{\ast\ast}^-, -m^-\leq y_2 \leq  \zeta_{1}^-(y_1)\},\\
     \mn_2^+&=\{0\leq y_1\leq \xi_{\ast}^+, \zeta_{1}^+(y_1)\leq y_2
     \leq m^+\}\cup
     \{\xi_{\ast}^+\leq y_1\leq \xi_{\ast\ast}^+, \zeta_{2}^+(y_1)\leq y_2
     \leq m^+\}.
     \end{aligned}
     \end{cases}
     \end{equation*}
      Given any point $ (y_1,y_2)\in  \mn_2 $, let $ \varphi^\pm_{i}(s;y_1,y_2) $ $ (i=1,2) $ be the characteristic curves associated with $ \lambda^\pm_i $ passing from the point
    $(y_1,y_2) $, i.e., defined by
     \begin{equation}\label{4-27}
     \begin{cases}
     \frac{\de \varphi^\pm_{i}(s;y_1,y_2)}{\de s}={\lambda^\pm_i(s,\varphi^\pm_{i}(s;y_1,y_2))}, \quad {\rm{for}} \quad 0\leq s\leq y_1,\\
     \varphi^\pm_{i}(y_1;y_1,y_2)=y_2.
     \end{cases}
     \end{equation}
     \par For any point $ (y_1,y_2)\in \mn_2^+ $, \eqref{4-27} implies that there exist unique $ \xi_{1}^+ $ and $\eta_{2}^+ $ such that the characteristic curve  associated with $ \lambda_1^+ $ starting from  $(y_1,y_2)$ has a unique intersection point
     $ (\xi_{1}^+,m^+) $ with the upper wall and the  characteristic curve  associated with  $ \lambda_2^+$ starting from  $(y_1,y_2)$   has a unique intersection point
     $ (0,\eta_{2}^+) $ with the  entrance $ \Sigma_{0}^+$. Thus we can regard  $ \xi_{1}^+ $ and $\eta_{2}^+ $ as functions of $ (y_1,y_2) $. Obviously, $ 0\leq \xi_{1}^+<y_1\leq \xi_{\ast\ast}^+ $ and $ 0\leq \eta_{2}^+\leq m^+ $. Moreover, due to the  reflection of the characteristic by the upper wall, there exists a unique $\eta_{1}^+ $ such that characteristic curve associated with  $ \lambda_2^+ $ passing from  $ (\xi_{1}^+,m^+) $  has a unique intersection point
     $ (0,\eta_{1}^+) $ with  the entrance $ \Sigma_{0}^+$ (See Fig 4).
       \par  Similarly, for any point $ (y_1,y_2)\in \mn_2^- $,  there exist unique $ \xi_{2}^- $ and $ \eta_{1}^-$ such that the characteristic curve corresponding to $ \lambda_2^- $ starting from $(y_1,y_2)$  has a unique intersection point
     $ (\xi_{2}^-,-m^-) $ with the  lower wall and the characteristic curve corresponding to $ \lambda_1^- $ starting from $(y_1,y_2)$ has a unique intersection point
     $ (0,\eta_{1}^-) $ with the entrance $ \Sigma_{0}^-$. Hence $ \xi_{2}^- $ and $\eta_{1}^- $ can be regarded as functions of $ (y_1,y_2) $ and   $ 0\leq \xi_{2}^-<y_1\leq \xi_{\ast\ast}^- $ and $ -m^-\leq \eta_{1}^-\leq 0 $. Furthermore,
     due to the  reflection of the characteristic by the lower wall, there exists a unique $\eta_{2}^- $ such that characteristic curve corresponding to $ \lambda_1^- $ passing from  $ (\xi_{2}^-,-m^-) $   has a unique intersection point
     $ (0,\eta_{2}^- ) $ with the entrance $ \Sigma_{0}^-$(See Fig 5).
     \begin{figure}

\begin{minipage}{0.44\textwidth}
  \centering
  \includegraphics[width=7cm,height=3.5cm]{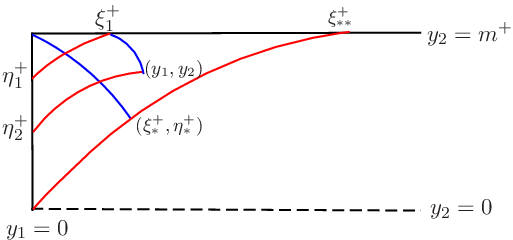}
  \caption{the  solution of  $({\mathbf{LP}})_2$ in $\mn_2^+$}
  \end{minipage}
 \begin{minipage}{0.44\textwidth}
  \centering
  \includegraphics[width=7cm,height=3.3cm]{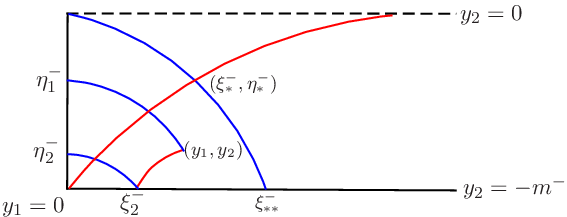}
  \caption{the  solution of  $({\mathbf{LP}})_2$ in $\mn_2^-$}
   \end{minipage}
   \end{figure}
\par For convenience, we only consider the  problem $({\mathbf{LP}})_2$ in $ \mn^+$.
    The  problem $({\mathbf{LP}})_2$ in $ \mn^-$ can be treated in the same way.
     Now, we consider the following problem:
       \begin{equation}\label{4-28}
    \begin{cases}
\p_{y_1} \Z^++diag(\lambda_1^+,\lambda_2^+)\p_{y_2}
    \Z^+-\mathbb{D}^+ \Z^+= {\bm F}^+,
    &\quad {\rm in} \quad\mn^+,\\
     \Z^+= \hat{\bm Z}_{0}^+,&\quad {\rm on} \quad \Sigma_{0}^+,\\
       \e Z_1^++  \e Z_2^+=2\arctan{(g^+)^{\prime}},&\quad {\rm on} \quad \Sigma_w^+.\\
     \end{cases}
    \end{equation}
    For any given point $(y_1,y_2)\in \mn_2^+$, along the characteristic curve  $\varphi_{i}^+(s;y_1,y_2)$, it holds that
     \begin{equation}\label{4-29}
     \begin{cases}
\begin{aligned}
    \e Z_1^+(y_1,y_2)
    &=  \e Z_1^+(\xi_{1}^+,m^+)+\int_{\xi_{1}^+}^{y_1}
    \left(\sum_{i=1}^2d_{1i}^+\e Z_i^++   F_1^+\right)(s,\varphi_{1}^+(s;y_1,y_2))\de s\\
   &= \left(2\arctan{(g^+)^{\prime}}- \e Z_2^+\right)(\xi_{1}^+,m^+)
  +\int_{\xi_{1}^+}^{y_1}
    \left(\sum_{i=1}^2d_{1i}^+\e Z_i^++  F_1^+\right)(s,\varphi_{1}^+(s;y_1,y_2))\de s,\\
 \e Z_2^+(y_1,y_2)
     &=  \hat Z_{20}^+(\eta_{2}^+)
  +\int_{0}^{y_1}
    \left(\sum_{i=1}^2d_{2i}^+\e Z_i^++   F_2^+\right)(s,\varphi_{2}^+(s;y_1,y_2))\de s.\\
    \end{aligned}
    \end{cases}
   \end{equation}
   Moreover,  the   characteristic curve corresponding to $ \lambda_2^+ $ passing from $(\xi_{1}^+,m^+)$ has a unique intersection point
     $ (0,\eta_{1}^+) $ with the entrance. Thus there holds
  \begin{equation}\label{4-30}
  \begin{aligned}
 \e Z_2^+(\xi_{1}^+,m^+)
    &=  \e Z_2^+(0,\varphi_{2}^+(0;\xi_{1}^+,m^+))
    +\int_{0}^{\xi_{1}^+}
    \left(\sum_{i=1}^2d_{2i}^+\e Z_i^++  F_2^+\right)(s,\varphi_{2}^+(s;\xi_{1}^+,m^+))\de s\\
   \quad &= \hat Z_{20}^+(\eta_{1}^+)
  +\int_{0}^{\xi_{1}^+}
    \left(\sum_{i=1}^2d_{2i}^+\e Z_i^++    F_2^+\right)(s,\varphi_{2}^+(s;\xi_{1}^+,m^+))\de s.\\
    \end{aligned}
    \end{equation}
   This, together with  \eqref{4-29}, implies that
    \begin{equation}\label{4-31}
    \begin{cases}
\begin{aligned}
\e Z_1^+(y_1,y_2)
&= 2\arctan{(g^+)^{\prime}}(\xi_{1}^+)
  +\int_{\xi_{1}^+}^{y_1}
    \left(\sum_{i=1}^2d_{1i}^+\e Z_i^++  F_1^+\right)(s,\varphi_{1}^+(s;y_1,y_2))\de s\\
    &\quad-\hat Z_{20}^+(\eta_{1}^+)
  -\int_{0}^{\xi_{1}^+}
    \left(\sum_{i=1}^2d_{2i}^+\e Z_i^++    F_2^+\right)(s,\varphi_{2}^+(s;\xi_{1}^+,m^+))\de s,\\
     \e Z_2^+(y_1,y_2
    &= \hat Z_{20}^+(\eta^+_{2})
  +\int_{0}^{y_1}
    \left(\sum_{i=1}^2d_{2i}^+\e Z_i^++   F_2^+\right)(s,\varphi_{2}^+(s;y_1,y_2))\de s.\\
   \end{aligned}
    \end{cases}
   \end{equation}
   \par It follows from  Picard iteration in \cite{LY85} that the integral equation \eqref{4-31} has a unique solution $ \Z^+\in C^{1}(\overline{\mn_1^+\cup\mn_2^+}) $. Hence there exists   a unique solution $ \Z^+\in C^{1}(\overline{\mn_1^+\cup\mn_2^+}) $ to the  boundary value problem \eqref{4-28}. Next,  the estimate of the solution can be established.
    \begin{proposition}
       For $ ( \hat{\bm Z}^-, \hat{\bm Z}^+)\in \mj_\delta $,  there exists a  positive constant  $ \mc_6^+ $ depending only on the data  such that  the solution $  \Z^+$ to the problem \eqref{4-28} satisfies
    \begin{equation}\label{4-32}
    \begin{aligned}
    \| \Z^+\|_{1,\alpha;\overline{\mn_1^+}\cup \overline{\mn_2^+}}
    &\leq \mc_6^+\left(\|(g^+)^\prime\|_{1,\alpha;[0,\xi_{\ast\ast}^+]}+\| \hat{\bm Z}_{0}^+\|_{1,\alpha; [0,m^+]}+\| {\bm F}^+\|_{1,\alpha;\overline{\mn^+}}\right)\\
&\leq \mc_6^+(\sigma_{cd}+\delta^2).
\end{aligned}
    \end{equation}
    \end{proposition}
    \begin{proof}
       Note that   $ \| \Z^+\|_{1,\alpha;\overline{\mn_1^+}}  $ has been obtained in Proposition \rm{4.2}. Thus we only consider the estimate of $
    \Z^+ $ in $ \overline{\mn_2^+} $.
    \par Define
   \begin{equation*}
     \begin{cases}
     \overline {\mn_{21}^+}(s)=\{0\leq w_1\leq s, \zeta_1^+(w_1)\leq w_2
     \leq m^+\}, \quad\quad {\rm {for}} \quad 0\leq s\leq \xi_{\ast}^+,\\
      \overline {\mn_{22}^+}(s)=\{\xi_{\ast}^+\leq w_1\leq s, \zeta_2^+(w_1)\leq w_2
     \leq m^+\},\quad\  {\rm {for}} \quad \xi_{\ast}^+\leq s\leq \xi_{\ast\ast}^+,
     \end{cases}
   \end{equation*}
   and set
     \begin{equation*}
     \begin{cases}
     \begin{aligned}
      \mz^+(s)&=\sum_{i=1}^2\sup _{ \bm (w_1,w_2)\in  \overline{\mn_{21}^+}(s)}|   \e Z_i^+(\bm w)|+\sum_{i=1}^2\sup _{ \bm w=(w_1,w_2)\in   \overline{\mn_{22}^+}(s)}|   \e Z_i^+(\bm w)|,\\
        D \mz^+(s)&=\sum_{i=1}^2\sup _{\bm w=(w_1,w_2)\in  \overline{\mn_{21}^+}(s)}|  D \e Z_i^+(\bm w)|+\sum_{i=1}^2\sup _{\bm w=(w_1,w_2)\in  \overline{\mn_{22}^+}(s)}|  D \e Z_i^+(\bm w)|,\\
   [D \mz^+]_\alpha(s)&=\sum_{i=1}^2\sup _{\substack{\bm w=(w_1,w_2)\in  \overline{\mn_{21}^+}(s)\\\w {\bm w} =(\w w_1,\w w_2)\in  \overline{\mn_{21}^+}(s)}}
     d^{-\alpha}(\bm w,\w {\bm w})|  D \e Z_i^+(\bm w)- D \e Z_i^+(\w {\bm w})|\\
     &\quad+\sum_{i=1}^2\sup _{\substack{\bm w=(w_1,w_2)\in  \overline{\mn_{22}^+}(s)\\\w {\bm w} =(\w w_1,\w w_2)\in  \overline{\mn_{22}^+}(s)}}
      d^{-\alpha}(\bm w,\w {\bm w})|  D \e Z_i^+(\bm w)- D \e Z_i^+(\w {\bm w})|.
       \end{aligned}
       \end{cases}
     \end{equation*}
     Obviously,
     \begin{equation}\label{4-33}
     \|
    \Z^+\|_{1,\alpha;\overline{\mn_2^+}}=\sup_{ 0\leq s\leq  \xi_{\ast\ast}^+} \mz^+(s)+\sup_{0\leq s\leq \xi_{\ast\ast}^+}
      D \mz^+(s)+\sup_{ 0\leq s\leq  \xi_{\ast\ast}^+} [D \mz^+]_\alpha(s).
    \end{equation}
     \par  Next,  we divide the proof into two steps.
   \par $\textbf{Step 1} $. The estimate of $ \|
    \Z^+\|_{0,0;\overline{\mn_2^+}} $. For $ 0\leq y_1\leq \xi_{\ast\ast}^+ $, it follows from \eqref{4-31}  that
    \begin{equation}\label{4-34}
    \mz^+(y_1)
    \leq C_{21}\left(\| \hat{ \bm Z}_{0}^+\|_{0,0; [0,m^+]}+\|  {\bm F}^+\|_{0,0;\overline{\mn^+ } }+\|(g^+)^\prime\|_{0,0;[0,\xi_{\ast\ast}]}\right)+C_{21}\int_{0}^{y_1}
    \mz^+(s)\de s,\\
   \end{equation}
      where  $  C_{21}>0 $   depends only on the data. Applying the  Gronwall inequality gets
    \begin{equation}\label{4-35}
   \| \Z^+\|_{0,0;\overline{\mn_2^+}}
        \leq C_{21}e^{C_{21}L}\left(\| \hat{ \bm Z}_{0}^+\|_{1,\alpha; [0,m^+]}+\|  {\bm F}^+\|_{1,\alpha;\overline{\mn^+} }+\|(g^+)^\prime\|_{1,\alpha;[0,\xi_{\ast\ast}]}\right).
      \end{equation}
         \par $\textbf{Step 2} $. The estimate of $ \|D
    \Z^+\|_{0,\alpha;\overline{\mn^+_2}} $. To do this,  differentiating the second equation in $\eqref{4-28} $ with respect to $y_1 $ and  $y_2 $ and integrating it along the characteristic curve $\varphi_{i}^+(s;y_1,y_2) $ to yield
    \begin{equation}\label{4-36}
    \begin{cases}
    \begin{aligned}
        \p_{y_j} \e Z_2^+(y_1,y_2)
        &=\p_{y_j} \e Z_2^+(0,\eta_2^+)
    +\int_{0}^{y_1}\sum_{i=1}^2d_{2i}^+\p_{y_j}\e Z_i^+(s,\varphi_{2}^+(s;y_1,y_2))\de s \\
    &\quad +\int_{0}^{y_1}\left(\p_{y_j}  F_2^+-\p_{y_j}\lambda_2^+\p_{y_2} \e Z_2^++\sum_{i=1}^2\p_{y_j}d_{2i}^+\e Z_i^+\right)(s,\varphi_{2}^+(s;y_1,y_2))\de s,\\
     \p_{y_j} \e Z_1^+(y_1,y_2)
     &=\p_{y_j} \e Z_1^+(\xi_1^+,m^+) +\int_{\xi_{1}^+}^{y_1}\sum_{i=1}^2d_{1i}^+\p_{y_j}\e Z_i^+(s,\varphi_{1}^+(s;y_1,y_2))\de s \\
    &\quad +\int_{\xi_{1}^+}^{y_1}\left(\p_{y_j}  F_1^+-\p_{y_j}\lambda_1^+\p_{y_2} \e Z_1^++\sum_{i=1}^2\p_{y_j}d_{1i}^+\e Z_i^+\right)(s,\varphi_{1}^+(s;y_1,y_2))\de s,\\
   \end{aligned}
   \end{cases}
  \end{equation}
     where $ \p_{y_1}d_{ij}^+=0, (i,j=1,2 ) $.
      Furthermore, by the boundary condition on the entrance $ \Sigma_0^+ $, one obtains
      \begin{equation}\label{4-37}
      \begin{cases}
      \begin{aligned}
      &\p_{y_2}\e Z_2^+(0,y_2)=(\hat{  Z}_{20}^+)^\prime(y_2),\\
      &\p_{y_1}\e Z_2^+(0,y_2)=\left(-\lambda_2^+(\hat{  Z}_{20}^+)^\prime+\sum_{i=1}^2d_{2i}^+\hat{  Z}_{i0}^++F_2^+\right)(0,y_2).\\
       \end{aligned}
       \end{cases}
       \end{equation}
      Meanwhile, it follows from the boundary condition on the upper wall $ \Sigma_w^+ $ that one has
       \begin{equation}\label{4-38}
      \begin{cases}
      \begin{aligned}
      \p_{y_1} \e Z_1^+(\xi_1^+,m^+)
      &=\frac{2(g^+)^{''}}{((g^+)^{\prime})^2+1}(\xi_1^+)-\p_{y_1} \e Z_2^+(\xi_1^+,m^+)\\
      &=\frac{2(g^+)^{''}}{((g^+)^{\prime})^2+1}(\xi_1^+)
      -\left(-\lambda_2^+(\hat{  Z}_{20}^+)^\prime+\sum_{i=1}^2d_{2i}^+\hat{  Z}_{i0}^++F_2^+\right)(0,\eta_1^+)\\
    &\quad -\int_{0}^{\xi_{1}^+}\left(\sum_{i=1}^2d_{2i}^+\p_{y_1}\e Z_i^++\p_{y_1}  F_2^+-\p_{y_1}\lambda_2^+\p_{y_2} \e Z_2^+\right)(s,\varphi_{2}^+(s;\xi_1^+,m^+))\de s,\\
    \p_{y_2} \e Z_1^+(\xi_1^+,m^+)
      &=\left(\frac{1}{\lambda_1^+}(-\p_{y_1} \e Z_1^++\sum_{i=1}^2d_{1i}\e{  Z}_{i}^++F_1^+)\right)(\xi_1^+,m^+).\\
     \end{aligned}
       \end{cases}
       \end{equation}
        \par  For $ 0\leq y_1\leq\xi_{\ast\ast}^+ $, it follows from \eqref{4-36}-\eqref{4-38}  that
     \begin{equation}\label{4-39}
     \begin{aligned}
     D \mz^+(y_1)
   \leq C_{22}\int_{0}^{y_1}D\mz^+(s)\de s
    +C_{22}\left(\|  \hat{ \bm Z}_{0}^+\|_{1,\alpha; [0,m^+]}+\| {\bm F}^+\|_{1,\alpha;\overline{\mn^+}}+\|(g^+)^{\prime}\|_{1,\alpha;[0, \xi_{\ast\ast}^+] }\right),\\
      \end{aligned}
       \end{equation}
       where  $  C_{22}>0 $   depends only on the data. Then applying the  Gronwall inequality derives
    \begin{equation}\label{4-40}
   \|D \Z^+\|_{0,0;\overline{\mn^+_2}}
        \leq C_{22} e^{C_{22}L}\left(\|  \hat{ \bm Z}_{0}^+\|_{1,\alpha; [0,m^+]}+\|{ \bm F}^+\|_{1,\alpha;\overline{\mn^+}}+\|(g^+)^{\prime}\|_{1,\alpha;[0, \xi_{\ast\ast}] }\right).
      \end{equation}
      \par To get the  estimate of $ [D \Z^+]_{0,\alpha;\overline{\mn^+_2}} $, we first establish the following estimate:
\begin{equation}\label{4-41}
     \sum_{i=1}^2\|\varphi_{i}^+(s;\cdot,\cdot)\|
_{1;\overline{\mn_1^+(s)\cup\mn_2^+(s)}}
+\|\xi_{1}^+\|_{1;\overline{\mn_2^+}}
+\sum_{i=1}^2\|\eta_{i}^+\|_{1;\overline{\mn_2^+}}
\leq C,
\end{equation}
where $ C>0 $ depends only the data but is independent of $ s $.
\par Differentiating  $\eqref{4-27} $ with respect to $ y_1 $ and $ y_2 $ to get
   \begin{equation*}
   \begin{cases}
    \begin{aligned}
 &\frac{\p \varphi_{i}^+(s;y_1,y_2)}{\p y_2}
 =e^{\int_{y_1}^{s}\p_{y_2}\lambda_i^+
 (\mu,\varphi_{i}^+(\mu;y_1,y_2))\de \mu},\\
 &\frac{\p \varphi_{i}^+(s;y_1,y_2)}{\p y_1}
 =-\lambda_i^+(y_1,y_2)e^{\int_{y_1}^{s}\p_{y_2}\lambda_i^+
 (\mu,\varphi_{i}^+(\mu;y_1,y_2))\de \mu}.\\
 \end{aligned}
 \end{cases}
      \end{equation*}
     Since $ \xi_{1}^+ $ and $\eta_{2}^+ $ can be regarded as functions of $ (y_1,y_2) $, hence it follows from \eqref{4-27} that
      \begin{equation}\label{4-42}
   \begin{cases}
   \begin{aligned}
   &m^+=y_2-\int_{\xi_1^+}^{y_1}
    {\lambda_1^+(s,\varphi_{1}^+
    (s;y_1,y_2))}\de s,\\
    &\eta_{2}^+=y_2-\int_{0}^{y_1}
    {\lambda_2^+(s,\varphi_{2}^+
    (s;y_1,y_2))}\de s,\\
     & \eta_{1}^+=m^+-\int_{0}^{\xi_{1}^+}
    {\lambda_2^+(s,\varphi_{2}^+
    (s;\xi_{1}^+,m^+))}\de s.
    \end{aligned}
    \end{cases}
    \end{equation}
    By differentiating the equations in $\eqref{4-42} $ with respect to $ y_1 $ and $ y_2 $, one can derive that
   \begin{equation*}
   \begin{cases}
   \begin{aligned}
 & \frac{\p \xi_{1}^+}{\p y_1}=-\frac{1}{\lambda_1^+(\xi_{1}^+,m^+)}
 \left(\int_{\xi_{1}^+}^{y_{1}}
    \lambda_1^+\p_{y_2}\lambda_1^+e^{\int_{y_1}^{s}\p_{y_2}\lambda_1^+
 (\mu,\varphi_{1}^+(\mu;y_1,y_2))\de \mu}\de s-\lambda_1^+(y_1,y_2)\right),\\
 &\frac{\p \xi_{1}^+}{\p y_2}=-\frac{1}{\lambda_1^+(\xi_{1}^+,m^+)}
 \left(1-\int_{\xi_{1}^+}^{y_{1}}
    \p_{y_2}\lambda_1^+
    e^{\int_{y_1}^{s}\p_{y_2}\lambda_1^+(\mu,\varphi_{1}^+(\mu;y_1,y_2))\de \mu}\de s\right),\\
     &\frac{\p \eta_{2}^+}{\p y_2}=1-\int_{0}^{y_1}\p_{y_2}\lambda_2^+
    e^{\int_{y_1}^{s}\p_{y_2}\lambda_2^+(\mu,\varphi_{2}^+(\mu;y_1,y_2)) \de \mu}\de s,\\
     &\frac{\p \eta_{2}^+}{\p y_1}=\int_{0}^{y_1}\lambda_2^+\p_{y_2}\lambda_2^+
    e^{\int_{y_1}^{s}\p_{y_2}\lambda_2^+(\mu,\varphi_{2}^+(\mu;y_1,y_2)) \de \mu}\de s-\lambda_2^+(y_1,y_2),\\
    & \frac{\p \eta_{1}^+}{\p y_j}=
 \left(\int_{0}^{\xi_{1}^+}
    \lambda_2^+\p_{y_2}\lambda_2^+e^{\int_{\xi_{1}^+}^{s}
    \p_{y_2}\lambda_2^+
 (\mu,\varphi_{2}^+(\mu;\xi_{1}^+,m^+))\de \mu}\de s-\lambda_2^+(\xi_{1}^+,m^+)\right) \frac{\p \xi_{1}^+}{\p y_j},j=1,2.
\end{aligned}
\end{cases}
 \end{equation*}
 Thus one infers
 \begin{equation*}
    \sum_{i=1}^2\|\varphi_{i}^+(s;\cdot,\cdot)\|
_{1;\overline{\mn_1^+(s)\cup\mn_2^+(s)}}
+\|\xi_{1}^+\|_{1;\overline{\mn_2^+}}
+\sum_{i=1}^2\|\eta_{i}^+\|_{1;\overline{\mn_2^+}}
\leq C.
\end{equation*}
\par For any two points $   \bm Q=(y_{1Q},y_{2Q}) $ and $ \bm P=(y_{1P},y_{2P})\in \mn^+_2 $,  let $ (\xi_{1}^+(\bm \ell),m^+) $ ($ \ell=P,Q $)   be the intersection point of the characteristic curve corresponding to $ \lambda_1^+ $ starting from point $(y_{1\ell},y_{2\ell})$ and the upper wall, and  let $ (0,\eta_{2}^+(\bm \ell)) $ be the intersection point of the characteristic curve corresponding to $ \lambda_2^+ $ starting from point $(y_{1\ell},y_{2\ell})$ and the entrance  $\Sigma_{0}^+ $.   Due to the  reflection of the characteristic by the upper wall,   the characteristic curve corresponding to $ \lambda_2^+ $ starting from point $ (\xi_{1}^+(\bm\ell),m^+) $  has a unique intersection point  $ (0,\eta_{1}^+(\bm \ell)) $ with the entrance $\Sigma_{0}^+ $. Without loss of generality, we assume that $ y_{1P}\leq y_{1Q} $. Then one has $ \xi_{1}^+(\bm P)\leq\xi_{1}^+(\bm Q) $.
 \par Note that
   \begin{equation*}
    \begin{aligned}
    &d^{-\alpha}(\bm P,\bm Q)\left(|D \e Z_1^+(\bm P)-D \e Z_1^+(\bm Q)|+|D \e Z_2^+(\bm P)-D \e Z_2^+(\bm Q)|\right)\\
   &\leq d^{-\alpha}(\bm P,\bm Q)\left(|\p_{y_1} \e Z_1^+(\bm P)-\p_{y_1} \e Z_1^+(\bm Q)|+|\p_{y_2} \e Z_1^+(\bm P)-\p_{y_2} \e Z_1^+(\bm Q)|\right)\\
    &\quad+d^{-\alpha}(\bm P,\bm Q)\left(|\p_{y_1} \e Z_2^+(\bm P)-\p_{y_1} \e Z_2^+(\bm Q)|+|\p_{y_2} \e Z_2^+(\bm P)-\p_{y_2} \e Z_2^+(\bm Q)|\right)\\
    &=K_{11}^++K_{12}^++K_{21}^++K_{22}^+.
       \end{aligned}
   \end{equation*}
   For convenience, we only estimate $ K_{11}^+ $, the estimate of other terms are similar to $ K_{11}^+ $.
   The straightforward calculations yield
       \begin{equation*}\label{4-49}
    \begin{aligned}
    K_{11}^+&\leq d^{-\alpha}(\bm P,\bm Q)|\p_{y_1}\e Z_1^+(\xi_1^+(\bm P),m^+)-\p_{y_1}\e Z_1^+(\xi_1^+(\bm Q),m^+)|\\
    &\quad+d^{-\alpha}(\bm P,\bm Q)
  \left| \int_{\xi_{1}^+(\bm Q)}^{y_{1Q}}\sum_{i=1}^2d_{1i}^+\p_{y_1}\e Z_i^+
   (s,\varphi_{1}^+(s;y_{1Q},y_{2Q}))\de s\right.\\
   &\qquad\qquad\qquad\qquad\left.-\int_{\xi_{1}^+(\bm P)}^{y_{1P}}\sum_{i=1}^2d_{1i}^+\p_{y_1}\e Z_i^+
   (s,\varphi_{1}^+(s;y_{1P},y_{2P}))\de s\right|\\
    &\quad+d^{-\alpha}(\bm P,\bm Q)
   \left|\int_{\xi_{1}^+(\bm Q)}^{y_{1Q}}  \left(\p_{y_1}F_1^+-\p_{y_1}\lambda_1^+\p_{y_2} \e Z_1^+\right)
   (s,\varphi_{1}^+(s;y_{1Q},y_{2Q}))\de s\right.\\
   \end{aligned}
   \end{equation*}
     \begin{equation}\label{4-43}
    \begin{aligned}
 &\qquad\qquad\qquad\qquad\left.-\int_{\xi_{1}^+(\bm P)}^{y_{1P}}\left(\p_{y_1}F_1^+-\p_{y_1}\lambda_1^+\p_{y_2} \e Z_1^+\right)
   (s,\varphi_{1}^+(s;y_{1P},y_{2P}))\de s\right|\\
    &= K_{111}^++K_{112}^++K_{113}^+.
     \end{aligned}
   \end{equation}
   For the first term, one derives
   \begin{equation*}
    \begin{aligned}
   &K_{111}^+=d^{-\alpha}(\bm P,\bm Q)|\p_{y_1}\e Z_1^+(\xi_1^+(\bm P),m^+)-\p_{y_1}\e Z_1^+(\xi_1^+(\bm Q),m^+)|\\
   &\leq \frac{|\xi_1^+(\bm P)-\xi_1^+(\bm Q)|^\alpha}{d^{\alpha}(\bm P,\bm Q)} \frac{\left|\frac{2(g^+)^{''}}{((g^+)^{\prime})^2+1}(\xi_1^+(\bm P))
   -\frac{2(g^+)^{''}}{((g^+)^{\prime})^2+1}(\xi_1^+(\bm Q))\right|}{|\xi_1^+(\bm P)-\xi_1^+(\bm Q)|^\alpha}\\
&\quad+\frac{|\left(-\lambda_2^+(\hat{  Z}_{20}^+)^\prime+\sum_{i=1}^2d_{2i}\hat{  Z}_{i0}^++F_2^+\right)(0,\eta_{1}^+(\bm P))-\left(-\lambda_2^+(\hat{  Z}_{20}^+)^\prime+\sum_{i=1}^2d_{2i}\hat{  Z}_{i0}^++F_2^+\right)(0,\eta_{1}^+(\bm Q))|}{|\eta_{1}^+(\bm P)-\eta_{1}^+(\bm Q)|^\alpha}\\
&\qquad \times \frac{|\eta_1^+(\bm P)-\eta_1^+(\bm Q)|^\alpha}{d^{\alpha}(\bm P,\bm Q)}\\
 &\quad+\int_{0}^{\xi_{1}^+(\bm P)}\frac{\left|\left(\sum_{i=1}^2d_{2i}^+\p_{y_1}\e Z_i^++\p_{y_1}  F_2^+-\p_{y_1}\lambda_2^+\p_{y_2} \e Z_2^+\right)(s,\varphi_{2}^+(s;\xi_1^+(\bm P),m^+))\right.}{|\varphi_{2}^+(s;\xi_1^+(\bm P),m^+)-\varphi_{2}^+(s;\xi_1^+(\bm Q),m^+)|^\alpha}\\
 &\qquad\qquad\qquad-\frac{\left.\left(\sum_{i=1}^2d_{2i}^+\p_{y_1}\e Z_i^++\p_{y_1}  F_2^+-\p_{y_1}\lambda_2^+\p_{y_2} \e Z_2^+\right)(s,\varphi_{2}^+(s;\xi_1^+(\bm Q),m^+))\right|}{|\varphi_{2}^+(s;\xi_1^+(\bm P),m^+)-\varphi_{2}^+(s;\xi_1^+(\bm Q),m^+)|^\alpha}\\
 &\quad\quad\quad\times \frac{|\varphi_{2}^+(s;\xi_1^+(\bm P),m^+)-\varphi_{2}^+(s;\xi_1^+(\bm Q),m^+)|^\alpha}{d^{\alpha}(\bm P,\bm Q)}
 \de s\\
 &\quad+\int_{\xi_{1}^+(\bm P^+)}^{\xi_{1}^+(\bm Q^+)}
   \left| \left(\sum_{i=1}^2d_{2i}^+\p_{y_1}\e Z_i^++\p_{y_1}  F_2^+-\p_{y_1}\lambda_2^+\p_{y_2} \e Z_2^+
 \right)(s,\varphi_{2}^+(s;\xi_1^+(\bm Q^+),m^+))\right|\de s\\
 &\leq Cd^{-\alpha}(\bm P,\bm Q)
    \int_{0}^{\xi_1^+(\bm P)}[D \mz^+]_\alpha(s) \de s\\
   &\quad+C\left(\|(g^+)^\prime\|_{1,\alpha;[0,\xi_{\ast\ast}]}+\|  \hat {\bm Z}_{0}^+\|_{1,\alpha;[0,m^+]}+\| \Z^+\|_{1,0;\overline{\mn^+_1}\cup\overline{\mn_2^+}}+\| { \bm F}^+\|_{1,\alpha;\overline{\mn^+} }\right).\\
 \end{aligned}
   \end{equation*}
      For the second term $ K_{112}^+ $,
   we divided two cases to estimate.
   \par  Case 1: $ y_{1P}\geq \xi_1^+(Q) $. For this case,
   it holds that
   \begin{equation*}
    \begin{aligned}
     K_{112}^+
   &\leq d^{-\alpha}(\bm P,\bm Q)
   \int_{\xi_{1}^+(\bm Q)}^{y_{1P}}\left|\sum_{i=1}^2d_{1i}^+\p_{y_1}\e Z_i^+
   (s,\varphi_{1}^+(s;y_{1Q},y_{2Q}))\right.\\
   &\qquad\qquad\qquad\qquad\quad\left.-
 \sum_{i=1}^2d_{1i}^+\p_{y_1}\e Z_i^+
   (s,\varphi_{1}^+(s;y_{1P},y_{2P}))\right|\de s\\
   &\quad+d^{-\alpha}(\bm P,\bm Q)\int_{y_{1 P}}^{y_{1Q}}\left|\sum_{i=1}^2d_{1i}^+\p_{y_1}\e Z_i^+
   (s,\varphi_{1}^+(s;y_{1Q},y_{2Q}))\right|\de s\\
    &\quad+d^{-\alpha}(\bm P,\bm Q)\int_{\xi_{1}^+(\bm P)}^{\xi_{1}^+( \bm Q)}\left|\sum_{i=1}^2d_{1i}^+\p_{y_1}\e Z_i^+
   (s,\varphi_{1}^+(s;y_{1P},y_{2P}))\right|\de s\\
   \end{aligned}
   \end{equation*}
   \begin{equation*}
    \begin{aligned}
  &\leq \int_{\xi_{1}^+(\bm Q)}^{y_{1P}}
     \frac{\sum_{i=1}^2\left|\p_{y_1}\e Z_i^+ (s,\varphi_{1}^+(s;y_{1Q},y_{2Q}))-\p_{y_1}\e Z_i^+ (s,\varphi_{1}^+(s;y_{1P},y_{2P}))\right|}
     {|\varphi_{1}^+(s;y_{1Q},y_{2Q})-\varphi_{1}^+(s;y_{1P},y_{2P})|^\alpha}
     \times\\
&\quad\quad\frac{|\varphi_{1}^+(s;y_{1Q},y_{2Q})
-\varphi_{1}^+(s;y_{1P},y_{2P})|
^\alpha}{d^{\alpha}(\bm P,\bm Q)}\left|\sum_{i=1}^2d_{1i}^+ (s,\varphi_{1}^+(s;y_{1Q},y_Q))\right| \de s\\
&\quad+C\| \Z^+\|_{1,0;\overline{\mn^+_1}\cup\overline{\mn_2^+}}\\
 &\leq C
    \int_{\xi_{1}^+(\bm Q)}^{y_{1P}}[D \mz^+]_\alpha(s) \de s+C\| \Z^+\|_{1,0;\overline{\mn^+_1}\cup\overline{\mn_2^+}}.\\
 \end{aligned}
 \end{equation*}
  \par  Case 2: $ y_{1P}<\xi_1^+(Q) $. For this case,
  one gets
   \begin{equation*}
    \begin{aligned}
     K_{112}^+
   &\leq d^{-\alpha}(\bm P,\bm Q)
  \left| \int_{\xi_{1}^+(\bm Q)}^{y_{1Q}}\sum_{i=1}^2d_{1i}^+\p_{y_1}\e Z_i^+
   (s,\varphi_{1}^+(s;y_{1Q},y_{2Q}))\de s
 -\int_{\xi_{1}^+(\bm Q)}^{y_{1Q}}\sum_{i=1}^2d_{1i}^+\p_{y_1}\e Z_i^+
   (s,m^+)\de s\right|\\
   &\quad+d^{-\alpha}(\bm P,\bm Q)
  \left| \int_{\xi_{1}^+(\bm P)}^{y_{1P}}\sum_{i=1}^2d_{1i}^+\p_{y_1}\e Z_i^+
   (s,m^+)\de s-\int_{\xi_{1}^+(\bm P)}^{y_{1P}}\sum_{i=1}^2d_{1i}^+\p_{y_1}\e Z_i^+
   (s,\varphi_{1}^+(s;y_{1P},y_{2P}))\de s\right|\\
   &\quad+d^{-\alpha}(\bm P,\bm Q)\left|\int_{\xi_{1}^+(\bm Q)}^{y_{1Q}}\sum_{i=1}^2d_{1i}^+\p_{y_1}\e Z_i^+
   (s,m^+)\de s- \int_{\xi_{1}^+(\bm P)}^{y_{1P}}\sum_{i=1}^2d_{1i}^+\p_{y_1}\e Z_i^+
   (s,m^+)\de s\right|\\
   &\leq \int_{\xi_{1}^+(\bm Q)}^{y_{1Q}}
     \frac{\sum_{i=1}^2\left|\p_{y_1}\e Z_i^+ (s,\varphi_{1}^+(s;y_{1Q},y_{2Q}))-\p_{y_1}\e Z_i^+ (s,\varphi_{1}^+(\xi_{1}^+(\bm Q);y_{1Q},y_{2Q}))\right|}
     {|\varphi_{1}^+(s;y_{1Q},y_{2Q})
     -\varphi_{1}^+(\xi_{1}^+(\bm Q);y_{1Q},y_{2Q})|^\alpha}
     \times\\
&\quad\quad\frac{|\varphi_{1}^+(s;y_{1Q},y_{2Q})
-\varphi_{1}^+(\xi_{1}^+(\bm Q);y_{1Q},y_{2Q})|
^\alpha}{d^{\alpha}(\bm P,\bm Q)}\left|d_{1i}^+ (s,\varphi_{1}^+(s;y_{1Q},y_Q))\right| \de s\\
&\quad+\int_{\xi_{1}^+(P)}^{y_{1P}}
     \frac{\sum_{i=1}^2\left|\p_{y_1}\e Z_i^+ (s,\varphi_{1}^+(s;y_{1P},y_{2P}))-\p_{y_1}\e Z_i^+ (s,\varphi_{1}^+(\xi_{1}^+(\bm P);y_{1P},y_{2P}))\right|}
     {|\varphi_{1}^+(s;y_{1P},y_{2P})
     -\varphi_{1}^+(\xi_{1}^+(\bm P);y_{1P},y_{2P})|^\alpha}
     \times\\
&\quad\quad\frac{|\varphi_{1}^+(s;y_{1P},y_{2P})
-\varphi_{1}^+(\xi_{1}^+(\bm P);y_{1P},y_{2P})|
^\alpha}{d^{\alpha}(\bm P,\bm Q)}\left|d_{1i}^+ (s,\varphi_{1}^+(s;y_{1Q},y_Q))\right| \de s\\
 &\quad+d^{-\alpha}(\bm P,\bm Q)\left(\int_{\xi_{1}^+(\bm P)}^{\xi_{1}^+(\bm Q)}\left|\sum_{i=1}^2d_{1i}^+\p_{y_1}\e Z_i^+
   (s,m^+)\right|\de s+ \int_{y_{1P}}^{y_{1Q}}\left|\sum_{i=1}^2d_{1i}^+\p_{y_1}\e Z_i^+
   (s,m^+)\right|\de s\right)\\
    \end{aligned}
   \end{equation*}
   \begin{equation*}
    \begin{aligned}
  &\leq \int_{\xi_{1}^+(\bm Q)}^{y_{1Q}} \frac{C|s-\xi_{1}^+(\bm Q)|}{|y_{1Q}-y_{1P}|}[D \mz^+]_\alpha(s) \de s
    +\int_{\xi_{1}^+(\bm P)}^{y_{1P}} \frac{C|s-\xi_{1}^+(\bm P)|}{{|y_{1P}-y_{1Q}|}}[D \mz^+]_\alpha(s) \de s+C\| \Z^+\|_{1,0;\overline{\mn^+_1}\cup\overline{\mn_2^+}}\\
    &\leq \int_{\xi_{1}^+(\bm Q)}^{y_{1Q}} C[D \mz^+]_\alpha(s) \de s
    +\int_{\xi_{1}^+(\bm P)}^{y_{1P}} C[D \mz^+]_\alpha(s) \de s+C\| \Z^+\|_{1,0;\overline{\mn^+_1}\cup\overline{\mn_2^+}}.\\
 \end{aligned}
   \end{equation*}
   Combining two cases, one gets
   \begin{equation}\label{4-44}
    K_{112}^+\leq C\int_{\xi_{1}^+(\bm P)}^{y_{1Q}} [D \mz^+]_\alpha(s)\de s +C\| \Z^+\|_{1,0;\overline{\mn^+_1}\cup\overline{\mn_2^+}}.
    \end{equation}
       For the last term  $ K_{113}^+ $, we can follow the estimate of $ K_{112}^+ $ to obtain
       \begin{equation*}\label{4-45}
    \begin{aligned}
    K_{113}^+
    &= d^{-\alpha}(\bm P,\bm Q)
   \left|\int_{\xi_{1}^+(\bm Q)}^{y_{1Q}}  \left(\p_{y_1}F_1^+-\p_{y_1}\lambda_1^+\p_{y_2} \e Z_1^+\right)
   (s,\varphi_{1}^+(s;y_{1Q},y_{2Q}))\de s\right.\\
    \end{aligned}
   \end{equation*}
     \begin{equation}\label{4-45}
    \begin{aligned}
 &\qquad\qquad\qquad\left.-\int_{\xi_{1}^+(\bm P)}^{y_{1P}}\left(\p_{y_1}F_1^+-\p_{y_1}\lambda_1^+\p_{y_2} \e Z_1^+\right)
   (s,\varphi_{1}^+(s;y_{1P},y_{2P}))\de s\right|\\
  &\leq C\int_{\xi_{1}^+(\bm P)}^{y_{1Q}} [D \mz^+]_\alpha(s)\de s +C\left(\| \Z^+\|_{1,0;\overline{\mn^+_1}\cup\overline{\mn_2^+}}+\| { \bm F}^+\|_{1,\alpha;\overline{\mn^+} }\right).
\end{aligned}
   \end{equation}
  \par Collecting the above estimates gives  that
   \begin{equation}\label{4-46}
   \begin{aligned}
   K_{11}^+&\leq C\int_{0}^{y_{1Q}} [D \mz^+]_\alpha(s)\de s\\
   &\quad+C\left(\|(g^+)^\prime\|_{1,\alpha;[0,\xi_{\ast\ast}]}+\|  \hat {\bm Z}_{0}^+\|_{1,\alpha;[0,m^+]}+\|    \bm F^+\|_{1,\alpha;\overline{\mn^+} }+\| \bm  \Z^+\|
_{1,0;\overline{\mn^+_1}\cup\overline{\mn_2^+}}\right),
\end{aligned}
\end{equation}
 where $C $ depends only on the data.
Similarly, one has
    \begin{equation}\label{4-47}
   \begin{aligned}
   &K_{12}^++K_{21}^++K_{22}^+\\
   &\leq
   C\left(\|(g^+)^\prime\|_{1,\alpha;[0,\xi_{\ast\ast}]}+\|  \hat {\bm Z}_{0}^+\|_{1,\alpha;[0,m^+]}+\|    \bm F^+\|_{1,\alpha;\overline{\mn^+} }+\| \bm  \Z^+\|
_{1,0;\overline{\mn^+_1}\cup\overline{\mn_2^+}}\right)\\
&\quad+C\int_{0}^{y_{1Q}}
     [D \mz^+]_\alpha(s)\de s.
     \end{aligned}
     \end{equation}
      Thus, for $ 0\leq y_{1P}\leq y_{1Q}\leq y_1\leq \xi_{\ast\ast}^+  $, it follows from \eqref{4-35}, \eqref{4-40}, \eqref{4-46}-\eqref{4-47} that
    \begin{equation}\label{4-48}
    \begin{aligned}
     &[D \mz^+]_\alpha(y_1)
     \leq C_{23}\int_{0}^{y_1}[D\mz^+]_\alpha(s)\de s+
  C_{23}\left(\|\hat{\bm Z}_{0}^+\|_{1,\alpha; [0,m^+]}+\|(g^+)^\prime\|_{1,\alpha;[0,\xi_{\ast\ast}]}+\|    \bm F^+\|_{1,\alpha;\overline{\mn^+} }\right).
    \end{aligned}
     \end{equation}
    By the Gronwall inequality, one derives
    \begin{equation}\label{4-49}
   [D \Z^+]_{0,\alpha;\overline{\mn_2^+}}
        \leq C_{23} e^{  C_{23}L}\left(\|\hat{\bm Z}_{0}^+\|_{1,\alpha; [0,m^+]}+\|(g^+)^\prime\|_{1,\alpha;[0,\xi_{\ast\ast}]}+\|    \bm F^+\|_{1,\alpha;\overline{\mn^+} }\right).
        \end{equation}
        Let $ \mc_6^+=C_{21} e^{ C_{21}L}+C_{22} e^{ C_{22}L}+C_{23}e^{ C_{23}L} $. Combining \eqref{4-35}, \eqref{4-40} and \eqref{4-49} yields that
        \begin{equation*}
        \begin{aligned}
    \| \Z^+\|_{1,\alpha;\overline{\mn_2^+}}
    &\leq \mc_{6}^+\left(\|\hat{\bm Z}_{0}^+\|_{1,\alpha; [0,m^+]}+\|(g^+)^\prime\|_{1,\alpha;[0,\xi_{\ast\ast}]}+\|    \bm F^+\|_{1,\alpha;\overline{\mn^+} }\right)\\
&\leq \mc_6^+(\sigma_{cd}+\delta^2).
\end{aligned}
    \end{equation*}
    \end{proof}
    \subsection{The solution to the boundary value problem  $({\mathbf{LP}} )$ in $ \mn_1\cup\mn_3$}\noindent
     \par In this subsection, we investigate  the  boundary value problem  $({\mathbf{LP}} )$ in $\mn_3$  with boundary conditions on the contact discontinuity $ y_2=0 $. The  boundary value problem  $({\mathbf{LP}} )$ in $\mn_3$ can be prescribed as the following:
  \begin{equation}\label{4-50}
  ({\mathbf{LP}})_3\quad  \begin{cases}
    \p_{y_1}  \Z^-+diag(\lambda_1^-,\lambda_2^-)\p_{y_2}
    \Z^--\mathbb{D}^- \Z^-
    = {\bm F}^-,
    &\quad {\rm in} \quad\mn^-,\\
    \p_{y_1} \Z^++diag(\lambda_1^+,\lambda_2^+)\p_{y_2}
    \Z^+-\mathbb{D}^+\Z^+=  {\bm F}^+,
    &\quad {\rm in} \quad\mn^+,\\
      \Z^-=  \hat{\bm Z}_{0}^-,&\quad {\rm on} \quad \Sigma_{0}^-,\\
      \Z^+=  \hat{\bm Z}_{0}^+,&\quad {\rm on} \quad \Sigma_{0}^+,\\
      \e Z_2^+- \e Z_1^-=  \e Z_2^-- \e Z_1^+ &\quad {\rm on} \quad \Sigma,\\
      \beta \e Z_2^+ +\alpha \e Z_1^-= \beta \e Z_1^++\alpha \e Z_2^--
      \mq, &\quad {\rm on} \quad \Sigma.\\
     \end{cases}
    \end{equation}
  \par  Let   $ (\xi_{\ast},\eta_{\ast})  $ be the  intersection point of the characteristic curves $ \zeta_{1}^+$ and $\zeta_{2}^-$.  Without loss of generality, we assume that $ \xi_{\ast}^-<\xi_{\ast}^+ $. Then $ \mn_3 $
  can be described as
     \begin{equation*}
    \begin{aligned}
     \mn_3&=\{0\leq y_1\leq  \xi_{\ast}^-, \zeta_{1}^-(y_1)\leq y_2\leq  \zeta_{2}^+(y_1)\}\cup \{\xi_{\ast}^-\leq y_1\leq \xi_{\ast}^+, \zeta_{2}^-(y_1)\leq y_2\leq  \zeta_{2}^+(y_1)\}\\
     &\quad \ \cup \{\xi_{\ast}^+\leq y_1\leq  \xi_{\ast}, \zeta_{2}^-(y_1)\leq y_2\leq  \zeta_{1}^+(y_1)\}.
     \end{aligned}
     \end{equation*}
     Denote $ \mn_3^\pm=\mn^\pm\cap \mn_3 $.  Given any point   $ \bm y=(y_1,y_2)\in  \mn_3^\pm $, let $ \varphi^\pm_{i}(s;y_1,y_2) $  $(i=1,2) $ be the characteristic curves associated with $ \lambda^\pm_i $  starting from point
    $(y_1,y_2) $, i.e., defined by
     \begin{equation}\label{4-51}
     \begin{cases}
     \frac{\de\varphi^\pm_{i}(s;y_1,y_2)}{\de s}={\lambda^\pm_i(s,\varphi^\pm_{i}(s;y_1,y_2))}, \quad {\rm{for}}\quad 0\leq s\leq y_1,\\
    \varphi^\pm_{i}(y_1;y_1,y_2)=y_2.
     \end{cases}
     \end{equation}
       \par For any point $ (y_1,y_2)\in \mn_3^+ $,  \eqref{4-51} implies that there exist unique $ \eta_{1}^+$ and $\xi_{cd}^+ $ such that the characteristic curve associated with  $ \lambda_1^+ $ starting from point $(y_1,y_2)$ has a unique intersection point
     $ (0,\eta_{1}^+) $ with  the entrance $ \Sigma_{0}^+ $  and the characteristic curve associated with $ \lambda_2^+ $ starting from point $(y_1,y_2)$  has a unique intersection point
     $ (\xi_{cd}^+,0) $ with the contact discontinuity $ y_2=0 $. Thus we can regard  $ \eta_{1}^+ $ and $\xi_{cd}^+ $ as functions of $ (y_1,y_2) $. Obviously, $ \xi_{cd}^+< y_1 $. Moreover, due to Rankine-Hugoniot
conditions on the contact discontinuity, there exist unique $\eta_{cd}^+$ and  $\eta_{2}^+$ such that the characteristic curve associated with $ \lambda_1^+ $ passing from point $ (\xi_{cd}^+,0) $ has a unique intersection point
     $ (0,\eta_{cd}^+) $ with  the entrance $ \Sigma_{0}^+ $ and the characteristic curve associated with $ \lambda_2^- $  passing from point $ (\xi_{cd}^+,0) $   has a unique intersection point  $ (0,\eta_{2}^+) $ with  the entrance $ \Sigma_{0}^- $ (See Fig 6).
     \par  Similarly, given any point $ (y_1,y_2)\in \mn_3^- $,  there exist unique $ \eta_{2}^- $ and $\xi_{cd}^- $ such that the characteristic curve corresponding to $ \lambda_2^- $ starting from point $(y_1,y_2)$ has a unique intersection point
     $ (0,\eta_{2}^-) $ with   the entrance $ \Sigma_{0}^- $ and the characteristic curve corresponding to $ \lambda_1^- $ starting from point $(y_1,y_2)$ has a unique intersection point
     $ (\xi_{cd}^-,0) $ with the contact discontinuity $ y_2=0 $. Hence $ \eta_{2}^- $ and $\xi_{cd}^- $ can be regarded as functions of $ (y_1,y_2) $. Obviously, $ \xi_{cd}^-< y_1 $. Moreover, due to Rankine-Hugoniot
conditions on the contact discontinuity, there exist unique $\eta_{cd}^- $ and  $\eta_{1}^- $ such that the characteristic curve corresponding to $ \lambda_2^- $  passing from point $ (\xi_{cd}^-,0) $ has a unique intersection point
     $ (0,\eta_{cd}^-) $ with the entrance $ \Sigma_{0}^- $  and the characteristic curve corresponding to $ \lambda_1^+ $  passing from point $ (\xi_{cd}^-,0) $  has a unique intersection point  $ (0,\eta_{1}^-) $ with  the entrance $ \Sigma_{0}^+ $ (See Fig 7).
      \begin{figure}

\begin{minipage}{0.44\textwidth}
  \centering
  \includegraphics[width=6.5cm,height=4cm]{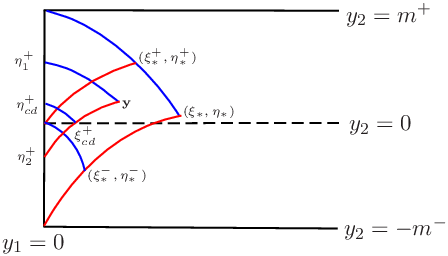}
  \caption{the  solution of  $({\mathbf{LP}})_3$ in $\mn_3^+$}
  \end{minipage}
 \begin{minipage}{0.44\textwidth}
  \centering
  \includegraphics[width=6.5cm,height=4cm]{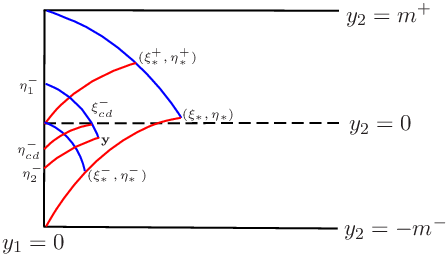}
  \caption{the  solution of  $({\mathbf{LP}})_3$ in $\mn_3^-$}
   \end{minipage}
   \end{figure}
   \par
    First, we show the existence of the solution. By the boundary conditions on $ \Sigma $, one derives
    \begin{equation}\label{4-52}
    \begin{aligned}
    & \e Z_2^+=\tau_1 \e Z_1^++\tau_3 \e Z_2^--\tau_4 \mq,\\
    & \e Z_1^-=\tau_2 \e Z_1^+-\tau_1 \e Z_2^--\tau_4 \mq,\\
    \end{aligned}
     \end{equation}
     where
     \begin{equation*}
    \tau_1=\frac{\beta-\alpha}{\beta+\alpha},\  \tau_2=\frac{2\alpha}{\beta+\alpha}, \ \tau_3=\frac{2\beta}{\beta+\alpha}, \ \tau_4=\frac{1}{\beta+\alpha}.
     \end{equation*}
     Then for any given   point $(y_1,y_2)\in \mn_3^+$, along the characteristic curve  $\varphi_{i}^+(s;y_1,y_2)$, it holds that
     \begin{equation}\label{4-53}
     \begin{cases}
\begin{aligned}
     \e Z_1^+(y_1,y_2)
    &= \hat  Z_{10}^+(\eta_{1}^+)+\int_{0}^{y_1}
    \left(\sum_{i=1}^2d_{1i}^+\e Z_i^++   F_1^+\right)(s,\varphi_{1}^+(s;y_1,y_2))\de s,\\
  \e Z_2^+(y_1,y_2)
    &=  \e Z_2^+(\xi_{cd}^+,0)+\int_{\xi_{cd}^+}^{y_1}
    \left(\sum_{i=1}^2d_{2i}^+\e Z_i^++  F_2^+\right)(s,\varphi_{2}^+(s;y_1,y_2))\de s\\
    &=\tau_1(\xi_{cd}^+) \e Z_1^+(\xi_{cd}^+,0)+\tau_3(\xi_{cd}^+) \e Z_2^-(\xi_{cd}^+,0)-\tau_4(\xi_{cd}^+) \mq\\
    &\quad+\int_{\xi_{cd}^+}^{y_1}
    \left(\sum_{i=1}^2d_{2i}^+\e Z_i^++  F_2^+\right)(s,\varphi_{2}^+(s;y_1,y_2))\de s\\
 &=\tau_1(\xi_{cd}^+)\left(  \hat Z_{10}^{+}(\eta_{cd}^+)+\int_{0}^{\xi_{cd}^+}
    \left(\sum_{i=1}^2d_{1i}^+\e Z_i^++    F_1^+\right)(s,\varphi_{1}^+(s;\xi_{cd}^+,0))\de s\right)\\
   &\quad+\tau_3(\xi_{cd}^+)\left(  \hat Z_{20}^-(\eta_{2}^+)+\int_{0}^{\xi_{cd}^+}
    \left(\sum_{i=1}^2d_{2i}^-\e Z_i^-+  F_2^-\right)(s,\varphi_{2}^-(s;\xi_{cd}^+,0))\de s\right)\\
   &\quad-\tau_4(\xi_{cd}^+)\mq++\int_{\xi_{cd}^+}^{y_1}
    \left(\sum_{i=1}^2d_{2i}^+\e Z_i^++  F_2^+\right)(s,\varphi_{2}^+(s;y_1,y_2))\de s.\\
   \end{aligned}
   \end{cases}
   \end{equation}
    Similarly, for any  point $(y_1,y_2)\in \mn_3^-$, along the characteristic curve  $\varphi_{i}^-(s;y_1,y_2)$, one gets
    \begin{equation}\label{4-54}
    \begin{cases}
  \begin{aligned}
  \e Z_2^-(y_1,y_2)
  &=\hat Z_{20}^-(\eta_{2}^-)+\int_{0}^{y_1}
    \left(\sum_{i=1}^2d_{2i}^-\e Z_i^-+  F_2^-\right)(s,\varphi_{2}^-(s;y_1,y_2))\de s,\\
     \e Z_1^-(y_1,y_2)
     &=Z_1^-(\xi_{cd}^-,0)
    +\int_{\xi_{cd}^-}^{y_1}
    \left(\sum_{i=1}^2d_{1i}^-\e Z_i^-+   F_1^-\right)(s,\varphi_{1}^-(s;y_1,y_2))\de s\\
  &=\tau_2(\xi_{cd}^-)\left(  \hat Z_{10}^+(\eta_{1}^-)+\int_{0}^{\xi_{cd}^-}
    \left(\sum_{i=1}^2d_{1i}^+\e Z_i^++  F_1^+\right)(s,\varphi_{1}^+(s;\xi_{cd}^-,0))\de s\right)\\
    &\quad-\tau_1(\xi_{cd}^-)\left(  \hat Z_{20}^-(\eta_{cd}^-)+\int_{0}^{\xi_{cd}^-}
    \left(\sum_{i=1}^2d_{2i}^-\e Z_i^-- F_2^-\right)(s,\varphi_{2}^-(s;\xi_{cd}^-,0))\de s\right)\\
    &\quad-\tau_4(\xi_{cd}^-) \mq++\int_{\xi_{cd}^-}^{y_1}
    \left(\sum_{i=1}^2d_{1i}^-\e Z_i^-+   F_1^-\right)(s,\varphi_{1}^-(s;y_1,y_2))\de s.\\
    \end{aligned}
    \end{cases}
   \end{equation}
   \par According to Picard iteration in \cite{LY85}, the integral equations \eqref{4-53} and \eqref{4-54} has a unique solution $ (\Z^+,\Z^-)\in (C^{1}(\overline{\mn_1^+\cup\mn_3^+})\times C^{1}(\overline{\mn_1^-\cup\mn_3^-})) $. Hence  boundary value problem \eqref{4-50} has a unique solution $ (\Z^+,\Z^-)\in (C^{1}(\overline{\mn_1^+\cup\mn_3^+})\times C^{1}(\overline{\mn_1^-\cup\mn_3^-})) $.
    Next,    we establish the  estimate for the solution.
       \begin{proposition}
       For $ (  \hat {\bm Z}^-, \hat {\bm Z}^+)\in \mj_\delta $,  there exists a  positive constant  $ \mc_7 $    depending only on the data  such that   the solution $( \Z^-,  \Z^+)$ to the problem $({\mathbf{LP}} )_3 $ satisfies
    \begin{equation}\label{4-55}
    \begin{aligned}
 &\| \Z^-\|_{1,\alpha;\overline{\mn_1^-}\cup\overline{ \mn_3^-}}+\| \Z^+\|_{1,\alpha;\overline{\mn_1^+}\cup\overline{ \mn_3^+}}\\
    &\leq \mc_7\left(\| \hat {\bm V}^-\|_{1,\alpha;[-m^-,0]}+\| \hat {\bm V}^+\|_{1,\alpha;[0,m^+]}+\sum_{I=\pm}\| {\bm F}^I\|_{1,\alpha;\overline{\mn^I}}\right)\\
       &\leq \mc_7(\sigma_{cd}+\delta^2),
       \end{aligned}
    \end{equation}
     where
     \begin{equation*}
 \hat {\bm V}^\pm=(\hat{\bm Z}_{0}^\pm,\hat B^\pm,\hat A^\pm)= \bm V_0^\pm -\bm V_b^\pm.
\end{equation*}
and $ \bm V_0^\pm, \bm V_b^\pm $  are defined in \eqref{3-29}.
    \end{proposition}
    The  detailed proof of this Proposition is  similar to the proof for Proposition 4.3, so we omit it here.
    \subsection{Proof of Theorem 4.1}
      \par   It follows from the estimates in Propositions 4.2, 4.3 and 4.4 that
        \begin{equation}\label{4-89}
    \begin{aligned}
     \sum_{I=\pm}\| \Z^I\|_{1,\alpha;\overline{\mn_1^I\cup \mn_2^I\cup \mn_3^I}}
    &\leq \mc_{1}^{\flat}\left(\| \hat {\bm V}^-\|_{1,\alpha;[-m^-,0]}+\| \hat {\bm V}^+\|_{1,\alpha;[0,m^+]}
    +\|(g^+)^\prime\|_{1,\alpha;[0,\xi_{\ast\ast}^+]}\right.\\
    &\quad\quad\quad \left.
    +\|(g^-)^\prime\|_{1,\alpha;[0,\xi_{\ast\ast}^-]}+\sum_{I=\pm}\| {\bm F}^I\|_{1,\alpha;\overline{\mn^I}}\right)\\
&\leq \mc_1^{\flat}(\sigma_{cd}+\delta^2),
\end{aligned}
        \end{equation}
        where  $ \mc_1^{\flat}>0 $ depends only on the data.
Let $ \xi_{\flat}^1=\min\{\xi_{\ast\ast}^-, \xi_{\ast\ast}^+\}  $  and $\tilde \mn_{\bm{1}}^{\pm}=\mn^{\pm}\cap\{0\leq y_1\leq \xi_{\flat}^1 \} $. Then we can follow the above proof to prove that \eqref{4-89} holds in $
\mn_4\cap\tilde\mn_{\bm{1}}^{\pm} $.  Then one has
  \begin{equation}\label{4-90}
    \begin{aligned}
     \sum_{I=\pm}\| \Z^I\|_{1,\alpha;\overline{\tilde\mn_{\bm{1}}^I}}
    &\leq \mc_{1}^{\flat}\left(\|\hat {\bm V}^-\|_{1,\alpha;[-m^-,0]}+\| \hat {\bm V}^+\|_{1,\alpha;[0,m^+]}
    +\|(g^+)^\prime\|_{1,\alpha;[0,\xi_{\ast\ast}^+]}\right.\\
    &\quad\quad\quad \left.
    +\|(g^-)^\prime\|_{1,\alpha;[0,\xi_{\ast\ast}^-]}+\sum_{I=\pm}\| {\bm F}^I\|_{1,\alpha;\overline{\mn^I}}\right)\\
&\leq \mc_1^{\flat}(\sigma_{cd}+\delta^2).
\end{aligned}
        \end{equation}
\par Let $ \xi_{\ast\ast\ast}^+ $  be the intersection point of the characteristic corresponding to $ \lambda_2^+ $ starting from $ (\xi_\flat^1,0) $ and the upper wall $ m^+ $. Let $ \xi_{\ast\ast\ast}^- $ be the intersection point of the characteristic corresponding to $ \lambda_1^- $ starting from $ (y_\flat^1,0) $ and the  lower wall $ -m^- $. Define $ \xi_\flat^{2}=\min\{\xi_{\ast\ast\ast}^+,\xi_{\ast\ast\ast}^-\} $ and $ \tilde\mn_{\bm{2}}^{\pm}=\mn^{\pm}\cap\{\xi_\flat^1\leq y_1\leq \xi_\flat^2\} $.  Then we can regard the line $ y_1=\xi_\flat^1 $ as the initial line $ y_1=0 $ and repeat the Step in the proof of Propositions 4.2 and 4.3 and 4.4 to obtain
\begin{equation}\label{4-91}
\begin{aligned}
 \sum_{I=\pm}\| \Z^I\|_{1,\alpha;\overline{\tilde\mn_{\bm{2}}^I}}
    &\leq \mc_{2}^{\flat}\left(\|\hat {\bm V}^-\|_{1,\alpha;[-m^-,0]}+\| \hat {\bm V}^+\|_{1,\alpha;[0,m^+]}
    +\|(g^+)^\prime\|_{1,\alpha;[0,\xi_{\ast\ast}^+]}\right.\\
    &\quad\quad\quad \left.
    +\|(g^-)^\prime\|_{1,\alpha;[0,\xi_{\ast\ast}^-]}+\sum_{I=\pm}\| {\bm F}^I\|_{1,\alpha;\overline{\mn^I}}\right)\\
    &\leq \mc_{2}^{\flat}\left(\|\hat {\bm V}^-\|_{1,\alpha;[-m^-,0]}+\| \hat {\bm V}^+\|_{1,\alpha;[0,m^+]}
    +\|(g^+)^\prime\|_{1,\alpha;[0,\xi_{\ast\ast\ast}^+]}\right.\\
    &\quad\quad\quad \left.
    +\|(g^-)^\prime\|_{1,\alpha;[0,\xi_{\ast\ast\ast}^-]}+\sum_{I=\pm}\| {\bm F}^I\|_{1,\alpha;\overline{\mn^I}}\right)\\
&\leq \mc_2^{\flat}(\sigma_{cd}+\delta^2),
    \end{aligned}
\end{equation}
where $ \mc_2^{\flat}>0 $ depends only on the data. Then  we can repeat this procedure  $ l $ times for the finite length $ L $  and define $ \xi_\flat^l $ and $ \tilde \mn_{\bm{l}}^{\pm} $ such that $ \mn^{\pm}=\bigcup_{1\leq k\leq l} \tilde \mn_{\bm{l}}^{\pm} $.   Summing all the estimates \eqref{4-90} together for $ k=1,\cdots, l $,  we finally obtain \eqref{4-1}.
\section{Solving the nonlinear value boundary problem $(\mathbf{P})$}\noindent
\par For given $ (\hat{ \bm Z}^-,\hat{ \bm Z}^-)\in \mj_\delta$,   the problem $({\mathbf{LP}} )$ has a unique solution  $( \Z^-,  \Z^+)$  satisfying
    \begin{equation}\label{5-1}
    \begin{aligned}
\sum_{I=\pm}\| \Z^I\|_{1,\alpha;\overline{\mn^I}}
    &\leq \mc_\flat\left(\| \bm V_{0}^--\bm V_{b}^-\|_{1,\alpha;[-m^-,0]}+\| \bm V_{0}^+-\bm V_{b}^+\|_{1,\alpha;[0,m^+]}\right.\\
    &\quad\quad\quad \left.+\|(g^+)^\prime\|_{1,\alpha;[0,L]}
    +\|(g^-)^\prime\|_{1,\alpha;[0,L]}+\sum_{I=\pm}\| {\bm F}^I\|_{1,\alpha;\overline{\mn^I}}\right)\\
&\leq \mc_\flat(\sigma_{cd}+\delta^2).
\end{aligned}
\end{equation}
Define a map
as follows
\begin{equation}\label{5-2}
\mt (\hat{  \bm Z}^-,\hat{ \bm Z}^+)=(\Z^-,\Z^+), \quad{\rm{ for \ each}} \ (\hat{  \bm Z}^-,\hat{ \bm Z}^+)\in \mj_\delta.
\end{equation}
If $ (\delta,\sigma_{cd}) $ satisfies
\begin{equation}\label{5-3}
\delta \leq \frac{1}{4}\min\left\{\frac{1}{\mc_\flat},\sigma_4\right\}, \quad {\rm{and}} \quad \sigma_{cd}\leq \frac{\delta}{4\mc_\flat},
\end{equation}
where $ \sigma_4 $ defined in Lemma 3.6. Then $ \mt $ maps $ \mj_\delta $ into itself. Observe that $  \mj_\delta $ is a compact convex set in $ C^{1,\alpha^\prime}(\overline{\mn^-})\times C^{1,\alpha^\prime}(\overline{\mn^+}) $ for $ 0<\alpha^\prime<\alpha $ and $ \mt $ is continuous. Indeed, the continuity of $ \mt $ can be proved by the argument for the estimate \eqref{5-1}. By the Schauder fixed point theorem, $ \mt $ has a fixed point in $  \mj_\delta$ provided that \eqref{5-3} holds.
\par Let $(  \hat{\bm Z}^-, \hat{\bm Z}^+)\in \mj_\delta$ be a fixed point of  $  \mj_\delta $. Then  $ (\bm Z^-,\bm Z^+):=
( \hat{\bm Z}^-, \hat{\bm Z}^+)+({\bm Z}_{b}^-,{\bm Z}_{b}^+) $ solves problem $ ({\mathbf{P}})$. Furthermore, it directly follows from  \eqref{5-1} and \eqref{5-3} that
 \begin{equation}\label{5-4}
  \sum_{I=\pm} \|\hat{\bm Z}^I\|_{1,\alpha;\overline{\mn^I}}
\leq \mc_\ast\left( \sum_{I=\pm} \|\hat{\bm Z}^I\|_{1,\alpha;\overline{\mn^I}}+\sigma_{cd}\right),
\end{equation}
where $ \mc_\ast $ depends only the data. Therefore, if $ (\delta,\sigma_{cd}) $ satisfies
\begin{equation}\label{5-5}
\delta \leq \frac{1}{4}\min\left\{\frac{1}{\mc_\ast},\frac{1}{\mc_\flat},\sigma_4\right\}, \quad {\rm{and}} \quad \sigma_{cd}\leq \frac{\delta}{4\mc_\flat},
\end{equation}
 the estimate \eqref{3-32} can be obtained from \eqref{5-4}.
\par To complete the proof of Theorem 3.5, it remains to prove the uniqueness of the solution to $\mathbf{Problem \ 3.4}$. Let $  (\bm Z_k^-,\bm Z_k^+)=( \hat{\bm Z}_k^-, \hat{\bm Z}_k^+)+({\bm Z}_{b}^-,{\bm Z}_{b}^+)$ $(k=1,2) $ be two solutions to problem $ ({\mathbf{P}})$, where $ \bm Z_k^\pm=(Z_{1k}^\pm,Z_{2k}^\pm) $. Set
\begin{equation*}
\bm \Phi^-= \hat{\bm Z}_1^--\hat{\bm Z}_2^-, \quad
\bm \Phi^+= \hat{\bm Z}_1^+-\hat{\bm Z}_2^+.
\end{equation*}
Then $(\bm \Phi^-,\bm \Phi^+) $ satisfies the following equation:
\begin{equation}\label{5-6}
     \begin{cases}
     \p_{y_1}\Phi^-+diag(\lambda_{11}^-,\lambda_{21}^-)\p_{y_2} \Phi^-\\
     =-diag(\lambda_{11}^--\lambda_{21}^-,\lambda_{21}^--\lambda_{22}^-)
     \p_{y_2} \hat{\bm Z}_2^-+
       {\bm H}_1^--{\bm H}_2^-,
    &\ {\rm in} \quad\mn^-,\\
     \p_{y_1}\Phi^++diag(\lambda_{11}^+,\lambda_{21}^+)\p_{y_2} \Phi^+\\
     =-diag(\lambda_{11}^+-\lambda_{21}^+,\lambda_{21}^+-\lambda_{22}^+)
     \p_{y_2} \hat{\bm Z}_2^++
       {\bm H}_1^+-{\bm H}_2^+,
    &\ {\rm in} \quad\mn^+,\\
    \bm \Phi^-= 0,&  \ {\rm on} \quad \Sigma_{0}^-,\\
    \bm \Phi^+= 0,&  \ {\rm on} \quad \Sigma_{0}^+,\\
      \Phi_1^-+ \Phi_2^-=0,&\ {\rm on} \quad \Sigma_w^-,\\
       \Phi_1^+ +\Phi_2^+=0,&\ {\rm on} \quad \Sigma_w^+,\\
       \Phi_1^-+ \Phi_2^-=\Phi_1^++ \Phi_2^+, &\ {\rm on} \quad \Sigma,\\
      (\alpha_1 \Phi_2^-+\beta_1  \Phi_1^+) -(\alpha_1  \Phi_1^-+\beta_1  \Phi_2^+)\\
      =( \alpha_2- \alpha_1)(\hat Z_{12}^--\hat Z_{22}^-)+( \beta_2- \beta_1)(\hat Z_{22}^+-\hat Z_{12}^+), &\ {\rm on} \quad \Sigma.\\
      \end{cases}
    \end{equation}
    Similar to the proof of Theorem 4.1, one has
    \begin{equation}\label{5-7}
    \sum_{I=\pm}\|\Phi^I\|_{0,\alpha;\overline{\mn^I}}
\leq \mc_{\sharp}\delta\sum_{I=\pm}\|\Phi^I\|_{0,\alpha;\overline{\mn^I}},
\end{equation}
where $ \mc_{\sharp} $ depends only on the data. Finally, we choose
\begin{equation}\label{5-8}
\delta = \frac{1}{4}\min\left\{\frac{1}{\mc_\sharp},\frac{1}{\mc_\ast},
\frac{1}{\mc_\flat},\sigma_4\right\}, \quad {\rm{and}} \quad \sigma_{3}= \frac{\delta}{4\mc_\flat}.
\end{equation}
Therefore, if $ \sigma_{cd}\leq \sigma_{3} $,  the inequality \eqref{5-7} implies $ \bm \Phi^{-}=\bm \Phi^{+}=0 $ so that $\mathbf{Problem \ 3.4}$ has a unique solution that satisfies \eqref{3-32}. This finishes the proof of Theorem 3.5.
\par {\bf Acknowledgement.} Weng is partially supported by National Natural Science Foundation of China  11971307, 12071359, 12221001.

\end{document}